\documentclass[12pt,reqno]{amsart}

\usepackage{amscd}

\newcommand{\N}{{\mathbb N}}
\newcommand{\Z}{{\mathbb Z}}

\newcommand{\C}{{\mathbb C}}
\newcommand{\R}{{\mathbb R}}
\renewcommand{\P}{{\mathbb P}}

\newcommand{\AAA}{{\mathcal A}}

\newcommand{\DD}{{\mathcal D}}

\newcommand{\KK}{{\mathcal K}}

\newcommand{\OO}{{\mathcal O}}

\newcommand{\QQ}{{\mathcal Q}}

\newcommand{\TT}{{\mathcal T}}
\newcommand{\UU}{{\mathcal U}}
\newcommand{\VV}{{\mathcal V}}

\newcommand{\ddd}{{\rm d}}
\newcommand{\www}{\widetilde}
\newcommand{\wwh}{\widehat}
\newcommand{\oooo}{\overline}
\newcommand{\uuuu}{\underline}
\newcommand{\iiii}{\infty}

\newcommand{\nmmm}{{\{0\}\times M}}

\newcommand{\immm}{{\{\infty\}\times M}}
\newcommand{\cmmm}{{\C\times M}}
\newcommand{\csmmm}{{\C^*\times M}}

\newcommand{\pmmm}{{\P^1\times M}}

\newcommand{\oomm}{\Omega}

\newcommand{\nnn}{\nabla}

\DeclareMathOperator{\End}{End}

\DeclareMathOperator{\id}{id}

\DeclareMathOperator{\rank}{rank}

\theoremstyle{plain}
\newtheorem{lemma}{Lemma}[section]
\newtheorem{theorem}[lemma]{Theorem}
\newtheorem{proposition}[lemma]{Proposition}
\newtheorem{corollary}[lemma]{Corollary}
\newtheorem{conjecture}[lemma]{Conjecture}

\theoremstyle{definition}
\newtheorem{definition}[lemma]{Definition}

\newtheorem{remark}[lemma]{Remark}
\newtheorem{remarks}[lemma]{Remarks}
\newtheorem{example}[lemma]{Example}

\begin{document}

\title[Hermitian metrics on $F$-manifolds]
{Hermitian metrics on $F$-manifolds}

\author[L. David and C. Hertling]{Liana David and Claus Hertling}

\address{Institute of Mathematics Simion Stoilow
of the Romanian Academy, Research Unit 4, Calea
Grivitei no. 21, Sector 1, Bucharest, Romania}
\email{liana.david@imar.ro}

\address{Lehrstuhl f\"ur Mathematik VI,
Universit\"at Mannheim, A5,6, 68131 Mannheim, Germany}
\email{hertling\char64 math.uni-mannheim.de}

\thanks{This work was supported by the
Alexander von Humboldt Foundation and the DFG grant He2287/4-1
(SISYPH). Partial financial support from a CNCS-UEFISCDI grant
PN-II-ID-PCE-2011-3-0362 is also acknowledged.}



\date{April 05, 2016}


\begin{abstract}
An $F$-manifold is complex manifold with a multiplication
on the holomorphic tangent bundle, which satisfies a certain integrability condition.
Important examples are Frobenius manifolds and especially
base spaces of universal unfoldings of isolated hypersurface
singularities. This paper reviews the construction
of hermitian metrics on $F$-manifolds from $tt^*$ geometry.
It clarifies the logic between several notions.
It also introduces a new {\it canonical} hermitian metric.
Near irreducible points it makes the manifold almost hyperbolic.
This holds for the singularity case and will hopefully lead
to applications there.
\end{abstract}

\maketitle


\tableofcontents


\section{Introduction}\label{s1}
\setcounter{equation}{0}

The notion of an $F$-manifold was introduced in the literature by Hertling and
Manin \cite{HM99} as a complex manifold  $M$  with a (fiber preserving)  associative, commutative, with unit field multiplication on $TM$, which satisfies a certain integrability
condition. 
There are various motivations to define and study this geometrical structure:  it is closely related to the notion of a Frobenius manifold, 
defined  by B. Dubrovin in 1991, as a geometrization of the so-called WDVV
(Witten-Dijkgraaf-Verlinde-Verlinde) equations. $F$-manifolds arise  naturally in the theory of 
meromorphic connections (and this is the view-point we shall adopt in this paper). 
Examples can be found  in integrable systems
and quantum cohomology as well.  This paper has three main purposes:

\medskip

(a)  Review the relation between $F$-manifolds (and stronger
structures, namely Saito and CV-structures) on the one hand, and meromorphic  connections on the other.
Review the existing notions
from  \cite{He03,Sa08,Li11,LS12}  (and sketch the results) of hermitian
metrics on $F$-manifolds and clarify their logic.

\medskip

(b)  Improve Theorem 4.5 in \cite{He03} and thus simplify
the definition of a $CDV$-structure there.

\medskip

(c) Introduce a new {\sf canonical hermitian metric } $h^{M}$  on
$M$, starting with a $CV$-structure on an abstract bundle $K \to M$, 
with the so-called unfolding condition (Definition \ref{t3.2}). 
Show that $h^{M}$ makes $M$ almost hyperbolic near points $t\in M$ where
$(T_tM,\circ_{t}, e_{t})$ is irreducible.

\medskip

The paper is structured as follows.   Section \ref{s2} is, in large part, an overview of the theory of certain meromorphic connections and the geometrical data they induce on the parameter space. We are interested in three classes of  connections: $(TEP)$-structures,
$(TLEP)$-structures and $(TERP)$-structures (see e.g.
\cite{He03}). A $(TEP)$-structure is a meromorphic connection $\nabla$ on a
holomorphic vector bundle
$H \to \mathbb{C}\times M$, with poles of 
Poincar\'e rank one along $\{ 0\}\times M$, and  a (holomorphic) flat pairing $P$, satisfying a set of natural conditions (see Definition \ref{t2.1}).
A $(TLEP)$-structure is essentially a $(TEP)$-structure which  extends  to $\mathbb{P}^{1}\times M$ and has  logarithmic poles along 
$ \{\infty\}\times M$. A $(TERP)$-structure is essentially  a $(TEP)$-structure together with a compatible real structure. 
In the first part of this section we  define 
 these notions, we discuss the (formal) isomorphy of semisimple $(TEP)$-structures and we recall how a (pure) $(TLEP)$-structure  and a (pure) $(TERP)$-structure on a bundle 
 $H\to \P^{1} \times M$
gives rise to a Saito structure $(K,C, \UU , g, \nabla^{r}, \VV )$,  respectivey to a 
CV-structure $(K,C, \UU , \kappa , h, D, \QQ )$, where  $K:= H\vert_{\{ 0\} \times M}$ (the definitions of Saito and CV-structures are recalled in
Definitions \ref{t2.8} and \ref{t2.10}).  Our original contribution in this section lies in Subsection \ref{s2.4}, where we consider 
a  Saito structure and a CV-structure, with the same underlying bundle, which share the same data $(C, \UU  ,g)$, and we ask to which extent they are isomorphic.  
In Definition \ref{var-equiv}  we propose various types of isomorphisms, in terms of the associated $(TEP)$-structures, and we discuss two particular cases in more 
detail -  the semisimple case and the conditions leading to the notion of a harmonic potential real Saito structure, introduced by Sabbah in \cite{Sa08} (see Definition \ref{harm}).

In Section \ref{s3} we study the geometry induced on the base space $M$
of an abstract holomorphic vector bundle $K \to M$, which underlies 
a Saito or a CV-structure, with the unfolding condition. By definition, 
the unfolding condition holds if there is  a (suitably chosen) holomorphic section 
$\zeta$ of $K$ (called a primitive section) such that
the map $TM \ni X \to - C_{X}(\zeta )\in K$ is an isomorphism, where $C$ is the Higgs field of the Saito or CV-structure. Using this isomorphism,  
the tangent bundle $TM$ inherits the data from $K$. The manifold $M$ becomes, in both cases
(i.e. when $K$ underlies a Saito or a CV-structure), 
an $F$-manifold with Euler field (see Definition \ref{t3.1} and Lemma \ref{t3.3}). Frobenius manifolds are recalled in Definition \ref{t3.6}. 
They arise from Saito structures. In particular, the  tangent bundle of any Frobenius manifold underlies a Saito structure.   
CV-structures whose underlying bundles are  tangent bundles  were studied in detail in \cite{He03},
where  the notion of a CDV-structure (a unification of a Saito and a CV-structure on a tangent bundle, with certain compatibility conditions, motivated by singularity theory), was introduced.
Here we prove Lemma \ref{help} and Theorem \ref{t3.8}, which are new.
Lemma \ref{help} shows that the operator $\QQ$ of a CV-structure
$(K,C, \UU , \kappa , h, D, \QQ )$,  for which the unfolding condition holds, is uniquely determined by 
$(K,C, \UU , h,D)$. Using this fact,  in Theorem \ref{t3.8}  we determine the expression of $\QQ$
when $K = TM.$ This leads to a simpler formulation of the 
defining conditions for a CDV-structure (see Definition \ref{t3.9} and 
Remarks \ref{t3.10}).

In Section  \ref{s4} we define the so-called {\sf canonical data}  on the base space $M$ of a CV-structure $(K\rightarrow M, C, \UU , \kappa , h, D, \QQ )$ with the unfolding 
condition   (see Definition \ref{t4.1}).
It consists of the system $(\circ , E, h^{M}, g^{M},\QQ^{M})$, formed by the multiplication $\circ$ and Euler field $E$ 
of the induced $F$-manifold structure on $M$  (see Lemma \ref{t3.3} mentioned above),  a hermitian (assumed to be non-degenerate, but not necessarily positive definite) 
metric $h^{M}$, a complex bilinear
(possibly degenerate) form $g^{M}$ and an endomorphism $\QQ^{M}.$ While in the previous sections (and in other references from the literature) the geometry of $M$
was inherited from $K$ using  the isomorphism $TM \ni X \to - C_{X} (\zeta )\in K$ determined by the choice 
of a primitive section, the system  $(\circ , E, h^{M}, g^{M},\QQ^{M})$ is independent of such a choice (this is the reason why it is called canonical).
The canonical data is obtained by identifying 
$TM$ with $F:= \{ C_{X},\ X\in TM\}$ 
(via the isomorphism  $X \to - C_{X}$)  and by considering $F$ as a subbundle of  
$\mathrm{End}(K)$ with induced structures. The geometry of $F$ is described in Lemmas 
\ref{t4.4} and \ref{t4.5}. Various compatibility relations between the tensors from the canonical data are collected in Proposition \ref{canonical-TM}.
In Theorem \ref{t4.6} we prove our main result about the canonical data, namely,  that if $h^{M}$ is positive definite, then  
it has non-positive holomorphic sectional
curvature and  the only obstruction
to be complex hyperbolic is the unit field $e$ of $M$.
In the main family of examples which we have in mind, $M$ is the base space
of a universal unfolding of an isolated hypersurface singularity.
We conclude with remarks on this case.

\section{$(TEP)$-structures and richer structures on
abstract bundles}\label{s2}
\setcounter{equation}{0}

Let us fix our  notation. 
Throughout the paper, $w$ is an integer and
$M$ is a complex manifold.
The sheaf of holomorphic 
functions on $M$ is ${\mathcal O}_{M}$, the sheaf of
holomorphic vector fields  is $\TT_M$,
the sheaf of $C^\iiii$ vector fields is
$\TT^\C_M=\TT^{1,0}_M\oplus\TT^{0,1}_M$. Similarly, 
the sheaf of holomorphic $k$-forms  is $\Omega^k_M$
and the sheaf of smooth $(p,q)$-forms is  $\Omega^{p,q}_{M}.$ 
We  denote by $L_{X}$ the Lie derivative on various tensor bundles on $M$, in the direction of the vector field $X\in {\mathcal T}_{M}$.   
When needed, vector fields on $M$, $\C$ or $\P^{1}$ will be identified (without stating explicitly) 
with their natural lifts to 
$\mathbb{C}\times M$ or $\mathbb{P}^{1}\times M$.
The coordinate on $\C$ is called $z$. 
The vector field $z\partial_{z}$ will be denoted by $U$.
We denote by $\oomm^1_{\C\times M} (\log(\nmmm))$ the sheaf of
holomorphic $1$-forms on $\C^{*}\times M$, which are logarithmic
along $\{ 0\} \times M$. Any such $1$-form is given locally by $\omega =
\frac{f_{0}}{z} dz + \sum_{i}f_{i}dt^{i}$, where $(t^{i})$ are
local coordinates on $M$ and $f_{i}\in {\mathcal O}_{\C \times
M}.$ 

For a holomorphic vector bundle $K$, we denote by $\mathcal O (K)$ the sheaf of holomorphic sections, by
$\Gamma (K)$ the sheaf of $C^{\infty}$-sections and by $\Omega^{k}(K)$
and $\Omega^{p,q} (K)$,  the sheaf of $K$-valued holomorphic $k$-forms, respectively, the sheaf of
$K$-valued smooth $(p,q)$-forms.  Always $p:\P^1\times M\to M$ is the natural projection and 
\begin{eqnarray}\label{2.1}
j:\P^1\times M\to \P^1\times M,\quad (z,t)\mapsto (-z,t).
\end{eqnarray}
A hermitian metric is a sesquilinear (= linear $\times$ conjugate
linear) hermitian non-degenerate form, but possibly with signature.
Often  a vector bundle $K$ will come with a non-degenerate complex bilinear
form $g$ or  with a non-degenerate hermitian metric $h$. We denote by $A^{*}$
and $A^{\flat}$ the $g$-adjoint, respectively the $h$-adjoint, of
any $A\in \mathrm{End}(K).$

\subsection{$(TEP)$-structures}\label{s2.1}

\begin{definition}\label{t2.1}
(a) Let $H\to\C\times M$ be a holomorphic vector bundle
and $\nnn$  a flat holomorphic  connection on the restriction
of $H$ to $\csmmm$. The connection $\nabla$ has {\sf a pole of Poincar\'e rank
$r\in\Z_{\geq 0}$} along $\nmmm$ if
\begin{eqnarray}\label{2.2}
\nnn:\OO(H)\to\frac{1}{z^r}\left(\oomm^1_{\C\times M}
(\log(\nmmm))\right)\otimes \OO(H).
\end{eqnarray}
A pole of Poincar\'e rank 0 is called a {\sf logarithmic pole}.

\medskip
(b) \cite[Definition 2.12]{He03}
A {\sf $(TE)$-structure} on $M$ is a tuple $(H\to\C\times M,\nnn)$ as in (a)
with a pole of Poincar\'e rank 1 along $\{ 0\} \times M.$

\medskip
(c) \cite[Definition 2.12]{He03} A {\sf $(TEP)$-structure} on $M$ of weight $w\in\Z$
(short: $(TEP)(w)$-structure) and rank $\mu$ is a tuple $(H\to \C\times
M,\nnn,P,w)$ such that $\mu=\rank H$, $(H\to\C\times M,\nnn)$ is a
$(TE)$-structure  and
\begin{eqnarray}\label{2.3}
P_{(z,t)}: H_{(z,t)}\times H_{(-z,t)}\to\C ,\ (z,t)\in \csmmm,
\end{eqnarray}
is a bilinear non-degenerate map which satisfies the following
conditions:

(i) it is $(-1)^{w}$-symmetric, i.e. for any  $(z,t)\in \C^*\times M$,
$$
P_{(z,t)}(a, b)= (-1)^{w} P_{(-z,t)}(b, a), \ a\in H_{(z,t)},\ b\in H_{(-z,t)}.
$$

(ii) for any $s, \tilde{s}\in {\mathcal O}(H)$,
$$
P(s, \tilde{s}):\C \times M \rightarrow \mathbb{C},\quad P(s,
\tilde{s}) (z,t):= P_{(z,t)} (s_{(z,t)}, \tilde{s}_{(-z,t)})
$$
belongs to $z^{w} {\mathcal O}_{\C \times M}$, 
and the maps 
\begin{equation}\label{p-concrete}
z^{-w} P_{(z,t)}: H_{(z,t)} \times H_{(-z,t)}\rightarrow \mathbb{C},\ (z,t)\in \C^{*}\times M
\end{equation}
extend to a non-degenerate pairing on $H_{(0, t)}$, for any $t\in M.$

(iii) $P$ is $\nabla$-flat, i.e.   for any $s, \tilde{s}\in {\mathcal O}(H)$ and $X\in {\mathcal T}_{M}$,
\begin{eqnarray*}
X (P(s, \tilde{s}) )= P ( \nabla_{X}(s),\tilde{s}) +
P (s, \nabla_{X}(\tilde{s}))\\
U (P(s, \tilde{s}) ) = P (\nabla_{U}(s), \tilde{s}) +
P (s,\nabla_{U}(\tilde{s})).
\end{eqnarray*}
\end{definition}

The non-degeneracy conditions from the above definition can be
written shortly by asking that  the map
\begin{eqnarray}\label{2.4} P:\OO(H)\otimes j^*\OO(H)\to
z^w\OO_\cmmm ,\ (s, j^{*}(\tilde{s}))\to P(s,\tilde{s})
\end{eqnarray}
is a non-degenerate pairing (non-degenerate along $\nmmm$ means
that $z^{-w}P$ is non-degenerate there).

\begin{lemma}\label{t2.2} \cite[Lemmas 2.4 and 2.14]{He03}
(a) Let $(H\to \cmmm,\nnn)$ be a $(TE)$-structure.
Define the holomorphic vector bundle $K:= H_{|\nmmm}$
and the maps 
\begin{equation}\label{2.5}
C=[z\nnn]:\OO(K)\to\oomm^1_M\otimes \OO(K),\  C_X[a]:=[z\nnn_Xa],
\end{equation}
where  $X\in \TT_M$ and $a\in\OO(H)$, 
and
\begin{eqnarray}\label{2.6}
\UU=[z\nabla_{U} ]:\OO(K)\to \OO(K).
\end{eqnarray}
Here $[\ ]$ means the restriction to $\nmmm$.  The maps $C_X$, $X\in \TT_M$, and $\UU$ are commuting, i.e. 
\begin{eqnarray}\label{2.7}
C\land C=0,\quad [C,\UU]=0,
\end{eqnarray}
where 
$$ (C\wedge C)_{X, Y} := C_{X}C_{Y} - C_{Y} C_{X},\ [C, {\mathcal U}]_{X}:= [ C_{X}, \mathcal U ],\quad X, Y\in {\mathcal T}_{M}. 
$$
In particular,  $C$ is a Higgs field.

(b) If there is a pairing $P$ which enriches the data to a
$(TEP)(w)$-structure, then
\begin{align}\label{2.8}
&g: \OO(K)\times \OO(K)\to\OO_M,\\
&g([a],[b]):= z^{-w}P(a,b)\mod z\OO_{\cmmm},\quad a,b\in\OO(H)\nonumber
\end{align}
is $\OO_M$-bilinear, symmetric, non-degenerate,
and  satisfies
\begin{align}\label{2.9}
&g(C_Xa,b) = g(a,C_Xb) \quad X\in\TT_M,a,b\in \OO(H),\\
& g(\UU a,b)= g(a,\UU b)\quad a,b\in\OO(K).
\label{2.10}
\end{align}
\end{lemma}

\begin{definition}\label{t2.3}
(a) A $(TEP)$-structure of rank $n$ on $M=\{\textup{pt}\}$ is {\sf semisimple}
if the eigenvalues $u_{1}, \cdots  u_{n}$  of $\UU$ are pairwise different.

\medskip

(b) A $(TEP)$-structure on an arbitrary manifold $M$ is {\sf semisimple}
if it is semisimple above any point $t\in M$. It is {\sf generically
semisimple} if it is semisimple above generic points $t\in M$.
\end{definition}

\begin{example}\label{t2.4}
(i)   Up to isomorphism (see the next subsection), there is a unique $(TEP)(w)$-structure 
of rank $1$ on $M= \C$, such that the eigenvalue of $\UU$  over $u\in \mathbb{C}$ is $u.$ It  has only the automorphisms $\pm\id$.
It is defined as follows. It is generated by a holomorphic section $s$, $
\OO(H)=\OO_{\C\times M}\cdot s$, with the properties
$$
\nonumber \ \nnn_{\frac{\partial}{\partial u}}(s) = -\frac{1}{z}\cdot s,\ 
\label{2.11}\nnn_{U}s  = (\frac{u}{z}+\frac{w}{2})\cdot s,\ P(s,s)(z, u)=i^w\cdot z^w.
$$
Let $\sigma$ be the (multivalued, if $w$ is odd) section of $H$, related to $s$ by 
$\sigma  =z^{-w/2}\cdot e^{u/z}\cdot s.$ Then $\sigma$ is $\nabla$-flat and  
$P(\sigma ,\sigma )=1.$

\medskip
(ii) Two $(TEP)$-structures $(H^{(i)}\to\C\times M^{(i)},
\nnn^{(i)},P^{(i)})$ ($i=1,2$) of weight  $w^{(1)}=w^{(2)}$
can be joined to form in an obvious way a
$(TEP)$-structure on $M=M^{(1)}\times M^{(2)}$ 
whose rank is $\rank H^{(1)}+\rank H^{(2)}$.

\medskip
(iii) $n$ rank 1 $(TEP)$-structures as in (i) can be joined
as in (ii) to form a semisimple $(TEP)$-structure of rank $n$
on the universal covering $N$ of
$\C^n-\{(u_1,....,u_n)\, |\,
\exists\ i\neq j\textup{ with }u_i=u_j\}$.
This is the universal semisimple $(TEP)$-structure with
trivial Stokes structure.
It has only the $2^n$ automorphisms which are in each factor
$\pm\id$.
\end{example}

\subsection{Equivalence of $(TEP)$-structures}

In this subsection we discuss  the formal equivalence (of germs) of $(TEP)$-structures.
We return to a related topic in Subsection \ref{s2.4}, where we consider various weaker equivalences between the 
meromorphic connections  associated to a Saito and a CV-structure.
Since we are interested in germ equivalences, we always assume that the 
bundles underlying our 
$(TEP)$-structures are (holomorphically) trivial. In particular, any formal automorphism between two such bundles is a power series expansion as in
(\ref{power-series}).

\begin{definition} \label{equiv-f}
Two $(TEP)(w)$-structures $(H^{(i)}\rightarrow \mathbb{C}\times M, \nabla^{(i)} , P^{(i)})$ 
($i=1,2$) are 
{\sf formally isomorphic} (respectively, {\sf isomorphic}), if there is  a  formal isomorphism (respectively, a holomorphic isomorphism)
\begin{equation}\label{power-series}
\Psi = \mathrm{id} + z\cdot A_{1} + z^{2}\cdot A_{2} +\cdots ,\quad A_{i}\in {\mathcal O}(\mathrm{Hom}(K_{1}, K_{2})),\  K_{i}:= H^{(i)}\vert_{\{ 0\} \times M}
\end{equation}
which satisfies
\begin{eqnarray}\label{ad1}
\nnn^{(2)}_{{X}}\circ \Psi =\Psi\circ \nnn^{(1)}_{{X}},\quad \nnn^{(2)}_{U}\circ \Psi = \Psi\circ \nnn^{(1)}_{U},\quad X\in {\mathcal T}_{M}
\end{eqnarray}
and 
\begin{eqnarray}\label{ad2}
{P}^{(2)}(\Psi a,\Psi b)=P^{(1)}(a,b)\quad a, b\in {\mathcal O}(H).
\end{eqnarray}

\end{definition}

The theorem of Hukuhara and Turrittin
and others on the formal decomposition of a holomorphic  vector bundle
on $(\C,0)\times (M,0)$ with a meromorphic connection
with an irregular pole along $\nmmm$, and
the theory of Stokes structures, take a particularly simple form, in the semisimple case.
(For a general discussion, see \cite{Ma83a}, \cite{Ma83b} or
\cite{Sa02}). Semisimple $(TEP)$-structures 
were studied sistematically in \cite[Chapter 8]{HS07}.
We will not review the general  theory here. We only state Theorem \ref{t2.5} (see below), which will be applied  in
Subsection \ref{s2.4}. Part (a) of this  theorem is a refinement of Theorem II.5.7
in \cite{Sa02}. Part (b) can be extracted from \cite{Ma83b},
up to the additional rigidity determined by the pairing.
A closely related statement can be found in \cite{LS12}. A detailed proof of Theorem \ref{t2.5} is postponed for a future paper
\cite{DH16},  where the equivalence of $(TEP)$-structures (not necessarily semisimple)  will
be studied systematically.

\begin{theorem}\label{t2.5}
(a) Let $(H^{(i)}\to\C\times M,\nnn^{(i)},P^ {(i)})$ ($i={1,2}$)
be two semisimple  $(TEP)(w)$-structures
on the same simply connected base space $M$, such
that $\UU^{(1)}$ and $\UU^{(2)}$ have everywhere the same
eigenvalues.
Then there are $2^n$ formal isomorphisms
\begin{eqnarray}\label{2.15}
\Psi&:&(\OO(H^{(1)})_{|(\C,0)\times M},\nnn^{(1)},P^{(1)})
\otimes\OO_M[[z]]\\
&&\to (\OO(H^{(2)})_{|(\C,0)\times M},
\nnn^{(2)},P^{(2)})\otimes\OO_M[[z]].\nonumber
\end{eqnarray}
If $H^{(2)}$ is the pull-back by a map $M\to N$ of 
the $(TEP)$-structure in Example \ref{t2.4} (iii)
(so it has trivial Stokes structure), then the $2^n$
isomorphisms arise from one isomorphism by composing it
with the $2^n$ automorphisms of $H^{(2)}$, mentioned  in Example \ref{t2.4}
(iii).

\medskip
(b) If one of the isomorphisms in (a) maps the Stokes structure
of the first $(TEP)$-structure to the Stokes structure of the
second $(TEP)$-structure (it is sufficient to have this
above one point $t_0\in M$), then it is holomorphic.
\end{theorem}

\subsection{Saito structures}\label{s2.2}

This subsection recalls a correspondence which underlies
one construction of Frobenius manifolds.
On one side is a meromorphic connection on a family of
trivial bundles on $\P^1\times M$ (a pure $(TLEP)$-structure),
on the other side is a differential geometric structure
on the restriction of the bundle to $\nmmm$ (a Saito structure).

\begin{definition}\label{t2.6} \cite[Definition 5.5]{He03}
A {\sf $(TLEP)$-structure} on $M$ of weight $w\in \Z$
(short: $(TLEP)(w)$-structure) is a tuple
$(H\to\pmmm,\nnn,P)$ such that its restriction
to $\cmmm$ is a $(TEP)(w)$-structure on $M$,
the pole of $\nabla$  along $\immm$ is logarithmic,
and the pairing $P$ of the
$(TEP)(w)$-structure extends to a non-degenerate pairing
\begin{eqnarray}\label{2.16}
P:\OO(H)\otimes j^*\OO(H)\to z^w\OO_\pmmm
\end{eqnarray}
on $\pmmm$. The $(TLEP)$-structure is {\sf pure} if the bundle $H$ is pure,
i.e.  each bundle $H_t\to \P^1\times \{t\}$ is  (holomorphically) trivial.
\end{definition}

This definition requires some comments.
The statement that $\nabla$ has a 
logarithmic pole along $\{\infty\} \times M$ 
means that $i^*(\nnn)$ has a logarithmic pole at $\{0\}\times M$
(Definition \ref{t2.1} (a)),
where $i(z,t)  = (\frac{1}{z}, t)$.
The statements on $P$ can be expressed more concretely as follows:
the map (\ref{p-concrete}) extends to $H_{(\infty , t)}$ (for any $t\in M$) and is non-degenerate there. Moreover, the extension is holomorphic. This means that for
any two sections $s$, $\tilde{s}$ of $H$, defined in  neighbourhood of $(\infty , t)$, the map $z^{-w} P(s, \tilde{s})\in {\mathcal O}_{\C \times M}$ belongs to 
${\mathcal O}_{\P^{1}\times M}.$

\begin{lemma}\label{t2.7} \cite[Lemma 5.3]{He03}
Let $(H\to\P^1\times M,\nnn,P,w)$ be a $(TLEP)(w)$-structure.
Let $\www K:=H_{|\immm}$ and denote by $[\ ]$ the restriction to $\immm$.
Define a holomorphic
connection $\nnn^{res}$ and a holomorphic endomorphism $\VV^{res}$
on $\www K$, by: for any $X\in \TT_{M}$ and $a\in \OO (H)$, 
\begin{eqnarray}\label{2.17}
\nnn^{res}_X [a] := [\nnn_X a],\quad 
\VV^{res}([a]):= - [\nnn_{U} a].
\end{eqnarray}
The connection $\nnn^{res}$ is flat and the endomorphism
$\VV^{res}$ is $\nnn^{res}$-flat:
\begin{eqnarray}\label{2.19}
(\nnn^{res})^2=0,\quad \nnn^{res}(\VV^{res})=0.
\end{eqnarray}
\end{lemma}

We now turn to the differential geometric view-point: Saito structures.

\begin{definition}\label{t2.8} \cite[Section 1.b]{Sa08},\cite[Definition 5.6]{He03}
A {\sf Saito structure} is a tuple $(K\to M,C,\UU,g,\nnn^{r},\VV)$,
where $K\to M$ is a holomorphic vector bundle, $C\in\Omega^{1}(\mathrm{End}(K))$, $\UU , \VV\in \OO (\mathrm{End}(K)),$ and
$g$ is a holomorphic symmetric non-degenerate bilinear form on $K$, such
that \eqref{2.7}, \eqref{2.9}, \eqref{2.10}, \eqref{2.19} (the
latter with $\nnn^{res}$, $\VV^{res}$ replaced by 
 $\nnn^{r}$,  $\VV$) and
\begin{eqnarray}\label{2.20a}
\nnn^{r}(C)&=&0,\\
\nnn^{r}(\UU)-[C,\VV]+C&=&0,\label{2.20b}\\
\nnn^{r}(g)&=&0,\label{2.21}\\
g(\VV a,b)+g(a,\VV b)&=&0,\ a,b\in \OO(K)
\label{2.22}
\end{eqnarray}
hold.   (In \cite[Definition 5.6]{He03} it is called {\sf Frobenius type
structure}.)
\end{definition}

In the above definition the $2$-form $\nabla^{r}(C) \in \Omega^{2}(M, \mathrm{End}(K))$ is defined by
\begin{equation}\label{saito-pot}
(\nabla^{r}(C))(X, Y) = \nabla^{r}_{X} (C_{Y}) - \nabla^{r}_{Y} (C_{X}) -  C_{[X,Y]},\ X, Y\in {\mathcal T}_{M},
\end{equation}
and the relation  (\ref{2.20a}) is called the {\sf potentiality condition}. 

\medskip

The relation between $(TEP)$-structures and Saito structures is expressed in the next theorem. 

\begin{theorem}\label{t2.9} \cite[Theorem 5.7]{He03}
Fix $w\in\Z$. There is a natural 1-1 correspondence between pure
$(TLEP)(w)$-structures $(H\to\pmmm,\nnn,P)$ and Saito structures
$(K\to M,C,\UU,g,\nnn^{r},\VV)$.
\end{theorem}

{\bf Sketch of the proof:} 
Consider a pure $(TLEP)(w)$-structure and define $K,C,\UU$ and $g$ as in Lemma \ref{t2.2}. Define $\www
K,\nnn^{res},\VV^{res}$ as in Lemma \ref{t2.7}. Due to the pureness, $K$
and $\www K$ are canonically isomorphic. 
Therefore, the connection  $\nnn^{res}$ and the endomorphism
$\VV^{res}+\frac{w}{2}$ 
of $\tilde{K}$ induce a connection $\nnn^{r}$ and an endomorphism $\VV$ of $K$.
Then $(K, C, \UU , g, \nabla^{r}, \VV )$ is a Saito structure. 

\medskip

Conversely, consider  a Saito structure $(K, C, \UU , g, \nabla^{r}, \VV )$.
Lift $C,\UU,\nnn^{r}$ and
$\VV$ canonically to $H:= p^{*}(K)$. Define the connection $\nnn$ on $H$ by
\begin{eqnarray}\label{2.23}
\nnn a &:=& \left(\nnn^{r} + \frac{1}{z}C +
\left(\frac{1}{z}\UU-\VV+\frac{w}{2}\right)\frac{\ddd z}{z}\right)(a),
\end{eqnarray}
for $a\in p^{-1}\OO(K)\subset\OO(H)$,
and the pairing $P$  by
\begin{eqnarray}\label{2.24}
P(a(z,t),b(-z,t))&:=& z^w\cdot g(a,b).
\end{eqnarray}
for $a,b\in K_t \cong H_{(z,t)}\cong H_{(-z,t)}$.
Then $(H, \nabla , P)$ is a pure $(TLEP)(w)$-structure.
\hfill$\Box$

\subsection{$tt^*$ geometry}\label{s2.3}

\noindent
$tt^*$ geometry was created in \cite{CV91,CV93}
and put into a framework  in \cite{He03}. It generalizes
variations of Hodge structures. 
Like before, there are two ways to approach $tt^{*}$-geometry: one by a meromorphic 
connection (a pure $(TERP)$-structure), the other by a
differential geometric structure (a $CV$-structure). 

A new feature is that the meromorphic connection can be 
defined in two ways: either on a holomorphic bundle on $\C\times M$
with an additional real structure on its restriction to $\C^*\times M$, or 
on a real analytic family of holomorphic bundles
on $\P^1$, parametrized by $M$.  Here we adopt the first way as the definition; the second way
will then be obtained  as an extension of the data to $\{\infty\}\times M$.
In fact, this is how $(TERP)$-structures arise in singularity theory:
the data on $\C\times M$ comes from geometry (by a 
Fourier-Laplace transformation of a Gau{\ss}-Manin connection) and 
the extension to $\{\infty\}\times M$ comes as an afterthought.
After presenting the basic facts on $(TERP)$-structures, we define the CV-structures and explain the correspondence between these and pure
$(TERP)$-structures.

\begin{definition}\label{t2.10a}
\cite[Definition 2.12]{He03}
A {\sf $(TERP)$-structure} of weight $w\in\Z$ is a tuple
$(H\to\C\times M,H'_\R,\nnn,P,w)$ such that 
$(H\to\C\times M,\nnn,P,w)$ is a $(TEP)$-structure of weight
$w$ and $H'_\R\to\C^*\times M$ is a real subbundle
of $H':=H|_{\C^*\times M}$, with the following properties:

\medskip

(a) $H_\R'$  is flat, i.e.  in a neighborhood of any point there is 
a basis of sections of $H_\R'$ which are $\nnn$-parallel and holomorphic.

\medskip 

(b) For any 
$(z,t)\in\C^*\times M$, 
$H_{(z,t)}=(H'_{\R})_{(z,t)}\oplus i(H'_{\R})_{(z,t)}$ and  
\begin{eqnarray}\label{2.25a}
P:(H'_{\R})_{(z,t)}\times (H'_{\R})_{(-z,t)}\to i^w\R.
\end{eqnarray}
\end{definition}

We remark that relation (\ref{2.25a}) is equivalent to
\begin{equation}\label{2.25a-1}
P (\kappa_{H} (s ), \kappa_{H}(\tilde{s })) = (-1)^{w} \overline{ P(s , \tilde{s })},\ s  , \tilde{s} \in {\mathcal O}(E).
\end{equation}

Following \cite{He03} (Lemma 2.14), we now describe the structure
at $\infty$ induced by a $(TERP)$-structure (the second way mentioned above).
Let $(H\to\C\times M,H'_\R,\nnn,P,w)$ be a $(TERP)$-structure
of weight $w$. 
As above, denote by $\kappa_H: H^{\prime} \rightarrow H^{\prime}$ the $\C$-antilinear involution with
$\kappa_H|_{H'_\R}=\id$. Define the antiholomorphic 
involution 
\begin{eqnarray}\label{2.25b}
\gamma:\P^1\times M\to\P^1\times M,\quad (z,t)\mapsto
(\frac{1}{\oooo{z}},t).
\end{eqnarray}
The points $(z,t)\in\C^*\times M$ and $\gamma(z,t)\in\C^*\times M$
are contained in the real half-line 
$\{(\zeta,t)\in\C^*\times M\, |\, \arg\zeta=\arg z\}$.
Define the $\C$-antilinear isomorphism 
\begin{eqnarray}\label{2.25d}
\tau : H^{\prime} \rightarrow H^{\prime},\  \tau_{(z,t)}:H_{(z,t)}&\to& H_{\gamma(z,t)},\quad (z,t)\in \mathbb{C}^{*} \times M,
\end{eqnarray}
where, for any $a\in H_{(z,t)}$,  $\tau_{(z,t)}(a):= [(\kappa_{H})_{(z,t)} (z^{-w}a)]^{\nabla} 
\in H_{\gamma (z,t)}$ is the 
$\nabla$-parallel transport 
of $(\kappa_{H})_{(z,t)} (z^{-w}a)\in H_{(z,t)}$ along this half line
(we always denote by a superscript $\nabla$ the $\nabla$-parallel transport of vectors along such lines).
Then $\tau$ satisfies $\tau^2=\id$. 
It maps a holomorphic section $\sigma$ of $H$
defined on $U\subset\C^*\times M$ to the section 
$\tau(\sigma)$ of $H$ defined on $\gamma(U)$, where
$$\tau(\sigma)(z,t):=\tau(\sigma(\gamma(z,t))).$$
This section is holomorphic in $z$ and real analytic in $t$.
The union of all sections $\tau(\sigma)$, where $\sigma$
are sections of $H$ near $\{0\}\times M$, defines an extension
of $H\to\C\times M$ to a vector bundle $\wwh{H}\to\P^1\times M$.
It is a real analytic family with respect to $t\in M$ of holomorphic 
vector bundles on $\P^1$.
The connection and the pairing of the $(TERP)$-structure extend to $\wwh{H}$ as follows.

\begin{lemma}\label{t2.10b}  
The pairing $z^{-w}P$ extends to a nondegenerate pairing
near $\{\infty\}\times M$.
The map $\tau$ extends to an automorphism of $\wwh{H}$,
which lifts $\gamma$.

The behaviour of the connection near $\{\infty\}\times M$ requires some 
diligence.
For fixed $t\in M$, it is a meromorphic connection on $\wwh{H}|_{\P^1\times\{t\}}$ 
with a pole of order $\leq 2$ at $\infty$.
If $\sigma$ is a holomorphic section on $U$ near $\{0\}\times M$, and
$X\in\TT_M$, then $\nnn_X\tau(\sigma)=0$, and $\nnn_{\oooo X}\tau(\sigma)$
can be written as a combination of sections $\tau(\sigma_1),...\tau(\sigma_\mu)$
where $\sigma_1,...,\sigma_\mu$ are a basis of $\OO(H)|_U$ and where the
coefficients are antiholomorphic in $t$ and meromorphic in $z$ with poles
of order $\leq 1$ along $\{\infty\}\times M$.
\end{lemma}

{\bf Proof:} To keep the text short, we only explain the statements about the pairing.
For any two holomorphic sections $\sigma$, $\tilde{\sigma}$ of $E$ 
in a neighbourhood of $(0,t)$, we can write
\begin{align*}
\nonumber & z^{-w}P( \tau (\sigma ), \tau (\tilde{\sigma }))(z,t) 
= z^{-w} P_{(z,t)}( \tau (\sigma )(z,t), \tau (\tilde{\sigma })(-z,t))\\
\nonumber & = (-z)^{w} P_{(z,t)} ( [\kappa_{H} \sigma ( \gamma (z,t) )]^{\nabla}, 
[\kappa_{H}  \sigma ( \gamma (-z,t))  ]^{\nabla}).
\end{align*}
From the flatness of $P$ and $H^{\prime}_{\mathbb{R}}$, we obtain
\begin{eqnarray*}
&&P_{(z,t)} ( [\kappa_{H} \sigma ( \gamma (z,t) )]^{\nabla}, 
[\kappa_{H}  \sigma ( \gamma (-z,t)) ]^{\nabla}) \\
&=& P_{\gamma (z,t)} ( \kappa_{H}\sigma ( \gamma (z,t)), 
\kappa_{H} \sigma (\gamma (-z,t))).
\end{eqnarray*}
Combining this relation with relation (\ref{2.25a-1}), we deduce that
$$
z^{-w}P( \tau (\sigma ), \tau (\tilde{\sigma }))(z,t) 
= (-1)^w \overline{ (1/\oooo{z})^{-w}
P_{\gamma (z,t)} (\sigma (\gamma (z,t)), \tilde{\sigma }(\gamma (-z,t)))}.
$$
For every fixed $t$, this function extends holomorphically to $z=\infty$.
\hfill$\Box$

\begin{definition} A $(TERP)$-structure is {\sf pure}  
if the bundle $\wwh{H}$  is pure, that means, each bundle 
$\wwh{H}|_{\P^1\times\{t\}}$, $t\in M$, is holomorphically trivial.
\end{definition}

We now turn to the differential geometrical side of $tt^{*}$-geometry: CV-structures.

\begin{definition}\label{t2.10} \cite[Definition 2.16]{He03}
A {\sf CV-structure} is a tuple
$(K,C,\UU,g,\kappa,h,D,\QQ)$ with the following properties:

\medskip 

(a) $K\to M$ is a holomorphic vector bundle, 
$C$ is a Higgs field, $\UU$ is a holomorphic endomorphism, and 
$g$ is  a holomorphic, symmetric, non-degenerate bilinear form, such that \eqref{2.7}, \eqref{2.9} and \eqref{2.10} hold.

\medskip

(b)  $\kappa:K\to K$ is a smooth fiberwise
$\C$-antilinear automorphism,  with $\kappa^2=\id$, and
\begin{eqnarray}\label{2.25}
g(\kappa a,\kappa b)&=&\oooo{g(a,b)}.
\end{eqnarray}

\medskip

(c) $h$ is a sesquilinear, hermitian and non-degenerate pairing, related to $g$ and $\kappa$ by
\begin{equation}\label{2.25f}
h= g(.,\kappa.)
\end{equation}
and whose Chern connection $D$  satisfies
\begin{align}\label{2.26}
D(\kappa)&=0,\\
D^{(1,0)}(C)&=0,\label{2.27}\\
D^{(1,0)}D^{(0,1)}+D^{(0,1)}D^{(1,0)}&= - ( C\kappa C\kappa  + \kappa C \kappa C) .\label{2.28}
\end{align}

(d)  $\QQ: K\to K$ is a smooth  endomorphism which  satisfies
\begin{align}
&g(\QQ a,b)= -g(a,\QQ b),\label{2.29}\\
&D^{(1,0)}(\UU)-[C,\QQ]+C=0,\label{2.30}\\
&D^{(1,0)}(\QQ)+[C,\kappa\UU\kappa]=0,\label{2.31}\\
&h(\QQ a,b) = h(a,\QQ b).\label{2.32}
\end{align}
The $CV$-structure is a {\sf $CV\oplus$-structure} if $h$ is positive
definite.
\end{definition}

In  relation (\ref{2.27}) $D^{(1,0)}(C)$ (often denoted $D(C)$, for simplicity)
is defined by (\ref{saito-pot})  (with $\nabla^{r}$ replaced by $D^{(1,0)}$).  
Relation (\ref{2.28}) can be written in the following equivalent way
\begin{equation}\label{alt-curv}
R^{D}(X, \bar{Y}) + [C_{X}, \kappa C_{Y}\kappa ]=0,\quad X, Y\in {\mathcal T}_{M}
\end{equation}
where $R^{D}$ is the curvature of $D$.

\begin{remark}\label{t2.11} 
(i) The above definition is a slightly modified (but equivalent) 
version of  the usual definition of CV-structures 
(see Definition 2.16 of \cite{He03}). In Definition 2.16 of \cite{He03} 
(as opposed to the above definition)
there is no holomorphic metric $g$.
But above $g$  is determined
by $h$ and $\kappa$ by $ g = h (\cdot , \kappa \cdot ).$ 
Instead,  Definition 2.16 of \cite{He03} 
contains a new endomorphism valued 1-form $\www{C}$. 
This is determined by $\kappa$ and $C$ by  
$\www{C}=\kappa C\kappa$.

\medskip
(ii) From  \cite[Lemma 2.18 (a)]{He03}, any CV-structure
is real analytic.
\end{remark}

The next theorem states the relation between $(TERP)$-structures and CV-structures.

\begin{theorem}\label{t2.12} \cite[Theorem 2.19]{He03}
Fix $w\in\Z$. There is a natural 1-1 correspondence between
pure $(TERP)(w)$-structures 
and $CV$-structures. 
\end{theorem}

{\bf Part of the proof:} 
To keep the text short,  we only   describe the construction of a $(TERP)(w)$-structure
from a $CV$-structure. Let 
$(K,C,\UU,g,\kappa,h,D,\QQ)$ be a CV-structure.
Lift $C,\UU,\QQ,D$ and $\kappa$ canonically
to $\wwh H := p^{*}K$ (considered as a  $C^{\infty}$-bundle). Define a connection $\wwh\nnn$ on $\wwh H$ by 
\begin{eqnarray}\label{2.33}
\wwh\nnn a :=
\left(D+\frac{1}{z}C+z\kappa C\kappa +
\left(\frac{1}{z}\UU-\QQ+\frac{w}{2}+z\kappa\UU\kappa\right)
\frac{\ddd z}{z}\right)(a),
\end{eqnarray}
for sections $a\in p^{-1}\Gamma (K)\subset \Gamma (\wwh
H)$.  Let $H:=\wwh H_{|\C\times M}$ with the
holomorphic structure  given as the kernel of the operator $\wwh\nnn^{(0,1)}$, where 
$$
\wwh\nnn^{(0,1)}_{\bar{X}}(a)= D^{(0,1)}_{\bar{X}}(a) + z \kappa C_{X} \kappa (a),\quad   X\in {\mathcal T}^{1,0}_{M},\  a\in p^{-1}\Gamma (H).
$$
Then $\wwh{\nabla}^{(1,0)}$ is  a holomorphic connection on $H$.
Define the pairing $\wwh{P}$ as in \eqref{2.24}.  
For any $(z,t)\in\P^1\times M$, define
\begin{eqnarray}\label{2.33b}
\tau_{(z,t)}:\wwh{H}_{(z,t)}&\to& \wwh{H}_{\gamma(z,t)},\ 
a\rightarrow \kappa(a).
\nonumber
\end{eqnarray}
Here we identify, in the canonical way, $K_{t}$ with $\widehat{H}_{(z,t)}$ 
and $\wwh{H}_{\gamma(z,t)}$.
For $(z,t)\in\C^*\times M$ define a real structure on $\wwh{H}_{(z,t)}$
with the $\C$-antilinear involution 
$$(\kappa_H)_{(z,t)}:\wwh{H}_{(z,t)}\to \wwh{H}_{(z,t)},$$
where, for any $a\in \wwh{H}_{(z,t)}$, $(\kappa_H)_{(z,t)}(a)$ is the
$\nnn$-parallel transport of $\tau(z^wa)\in \wwh{H}_{\gamma(z,t)}$
along the real half line $\{(\zeta,t)\in\C^*\times M\, |\, \arg\zeta=\arg z\}$.
Then $(H\rightarrow \mathbb{C}\times M, \widehat{\nabla}, \widehat{P},\kappa_{H})$ is a $(TERP)(w)$-structure. 
\hfill$\Box$

\subsection{Compatibilities between CV and Saito structures}\label{s2.4}
Let $(K,C ,\UU  ,g,\nnn^{r},\VV)$ and $(K,C,\UU,g,\kappa,h,D,\QQ)$
be a Saito and a CV-structure, which share the same data
$(K\rightarrow M, C, \UU , g).$ 
Let $(p^{*}K, \nabla , P)$ be the $(TEP)(w)$-structure determined 
by the Saito structure $(K,C ,\UU  ,g,\nnn^{r},\VV)$. 
Similarly, let  $(p^{*}K, \widehat{\nabla}, \widehat{P})$ 
be the $(TEP)(w)$-structure which underlies the $(TERP)(w)$-structure 
determined by the $CV$-structure $(K,C,\UU,g,\kappa,h,D,\QQ)$.
We ask to which extent these two $(TEP)$-structures  
(or rather their restrictions to $(\C , 0)\times M$) are isomorphic.

From Theorem \ref{t2.5}, 
when $\UU$ is semisimple 
(with different eigenvalues) and $M$ is simply connected, 
$(p^{*}K, \nabla , P)$  and  $(p^{*}K, \widehat{\nabla}, \widehat{P})$
are formally isomorphic. 
(But holomorphic isomorphism is stronger, 
as  it requires that they have the same Stokes structures).
Our expectations in cases different from the semisimple case will be discussed
in Conjecture \ref{t3.5} (in Subsection \ref{s3.1}).

The next definition is weaker than Definition \ref{equiv-f}:  
the $A(\infty )$-condition from Definition \ref{var-equiv} 
is equivalent to the existence of a formal isomorphism  between   
the two $(TEP)$-structures,  as in  Definition \ref{equiv-f}.

\begin{definition} \label{var-equiv} (a) The $(1,0)$-parts  $\nabla^{(1,0)}$ and $\wwh{\nabla}^{(1,0)}$ 
are  {\sf formally isomorphic up to order  $(k-1)$}  if and only if there is a formal automorphism 
\begin{eqnarray}\label{psi-form}
\Psi := \mathrm{id} + z\cdot A_{1} + z^{2}\cdot A_{2} + \cdots ,\quad A_{i} \in \Gamma ( \mathrm{End}(K))
\end{eqnarray}
such that  
 \begin{eqnarray*}
 \wwh{\nabla}_{X}  (\Psi (a)) = \Psi \nabla_{X}(a),\ \wwh{\nabla}_{U}(\Psi (a)) = \Psi \nabla_{U}(a)
 \end{eqnarray*}
 hold  modulo  $z^{k}{\mathcal O}[[z]]$, for any $X\in {\mathcal T}_{M}$ and $a\in p^{-1}\Gamma (K).$ If, in addition, 
 $$
 z^{-w-1} P(\Psi (a), \Psi (b)) = z^{-w-1} P(a,b)\ \mathrm{mod}\ z^{k} {\mathcal O}[[z]],\quad \forall \ a, b\in p^{-1}{\Gamma}(K)
 $$
 then the pairs $(\nabla^{(1,0)}, P)$ and $(\wwh{\nabla}^{(1,0)},\wwh{P})$ 
are  {\sf formally isomorphic up to order  $(k-1)$} .

 \medskip
  
(b) The holomorphic structures $\nabla^{(0,1)}$ and $\widehat{\nabla}^{(0,1)}$ are {\sf  formally isomorphic up to order $(k-1)$ }
if and only if there is  a formal automorphism $\Psi$ as in (\ref{psi-form}),  such that, for any $X\in {\mathcal T}^{1,0}_{M}$ and $a\in p^{-1} {\Gamma}(K)$, 
\begin{eqnarray}\label{2.35}
\wwh\nnn_{\bar{X}} ( \Psi (a))= \Psi ( \nnn_{\bar{X}}(a))\  \mathrm{mod}\  z^{k}{\mathcal O}[[z]].
\end{eqnarray}

(c) The $(TEP)$-structures $(p^{*}(K), \nabla , P)$ and $(p^{*}(K), \wwh{\nabla}, \wwh{P})$  are {\sf formally isomorphic up to order $(k-1)$}
(or the {\sf $A(k-1)$-condition} holds)  if  and only if the pairs
$(\nabla^{(1,0)}, P)$ and $(\wwh{\nabla}^{(1,0)},\wwh{P})$ 
are  formally isomorphic up to order  $(k-1)$ and the holomorphic structures $\nabla^{(0,1)}$ and $\widehat{\nabla}^{(0,1)}$
are formally isomorphic up to order $k$ (by the same formal isomorphism).

\medskip

(d) The $(TEP)$-structures $(p^{*}(K) , \nabla , P)$ and $(p^{*}(K), \wwh{\nabla}, \wwh{P})$  are {\sf isomorphic}
(or the  {\sf $A(hol)$-condition} holds)  if  and only if the $A(\infty )$-condition holds and 
the formal isomorphism is convergent.
 \end{definition}

In this section we are concerned with formal equivalences. (The holomorphic equivalence will be discussed in the next section,
 in relation with Conjecture \ref{t3.5}). Since the Saito and CV-structures share the same data $(K, C, \mathcal U , g)$, the condition $A(-1)$ holds
(easy check). As discussed in \cite{Sa08}, 
the condition $\nabla^{(0,1)}$ and $\widehat{\nabla}^{(0,1)}$-formally isomorphic up to order $1$ leads to the so-called harmonic potential Higgs structures. 
We now recall the notion of harmonic potential real Saito bundle 
(defined for the first time in  \cite[Section 1.d]{Sa08}) and we show that it is equivalent to the $A(0)$-condition above.

\begin{definition}\label{harm}
The union of a Saito structure 
$(K,C, \UU , g, \nabla^{r}, \VV )$ and a $CV$-structure
$(K,C, \UU , g, \kappa , h, D, \QQ )$, which share the same data $(K,C,\UU,g)$, 
forms a  {\sf harmonic potential real Saito structure}, if there is 
a $g$-symmetric (not necessarily holomorphic) endomorphism $A\in \Gamma (\mathrm{End}(K))$ (called the {\sf potential}), such that 
\begin{eqnarray}\label{def-harm}
D=\nnn^r- [A^{\flat},C],\quad
\QQ=\VV-[\UU , A^{\flat}],\quad D^{(0,1)} ( A^{\flat}) = \kappa C \kappa .
\end{eqnarray}
\end{definition}

\begin{remark}\label{simpler}
Consider a Saito structure and a CV-structure, 
with the same data $(K\rightarrow M, C, \UU , g).$ 
We  claim that if $A$ is any (smooth) endomorphism of $K$, which satisfies
(\ref{def-harm}), then its $g$-symmetric part $\frac{1}{2} ( A + A^{*})$ also satisfies (\ref{def-harm})
(in particular, the $g$-symmetry condition from the above definition is not essential).
To prove the claim, 
we suppose that such an $A$ is given.
The $h$-hermitian adjoint $\tilde{A}:= A^{\flat}$ satisfies
\begin{eqnarray}\label{def-harm-1}
D=\nnn^r- [\tilde{A},C],\quad
\QQ=\VV-[\UU , \tilde{A}],\quad D^{(0,1)} ( \tilde{A}) = \kappa C \kappa .
\end{eqnarray}
Since $D(g) = \nabla^{r}(g)=0$, $[\tilde{A}, C_{X}]$ is $g$-skew-symmetric, for any $X\in TM.$ 
Using that $C_{X}$ is $g$-symmetric, we obtain that $[C_{X}, \tilde{A} - \tilde{A}^{*}]=0.$ A similar argument (which uses the fact that both 
$\QQ$ and $\VV$ are $g$-skew-symmetric and $\UU$  is  $g$-symmetric) shows that $[\UU , \tilde{A} - \tilde{A}^{*}]=0.$
We proved that if the first two relations (\ref{def-harm-1}) are satisfied by $\tilde{A}$, then they are also satisfied by its $g$-symmetric part
$\frac{1}{2} ( \tilde{A} + \tilde{A}^{*}).$ We consider now the third relation (\ref{def-harm-1}).
It is equivalent to 
\begin{equation}\label{relation-1}
\bar{X} g (\tilde{A}(s_{1}), s_{2} ) = g( \kappa C_{X}\kappa (s_{1}), s_{2}),\quad  s_{1},s_{2}\in {\mathcal O} (K),\ X\in {\mathcal T}_{M},  
\end{equation}
or (by interchanging $s_{1}$ with $s_{2}$) to 
\begin{equation}\label{relation-2}
\bar{X} g (s_{2}, \tilde{A}^{*}(s_{1}) ) = g( \kappa C_{X}\kappa (s_{2}), s_{1}),\quad  s_{1},s_{2}\in {\mathcal O} (K),\ X\in {\mathcal T}_{M}.
\end{equation}
Using that 
$g(\kappa C_{X} \kappa (s_{1}), s_{2})$ is symmetric in $s_{1}$ and $s_{2}$ (being equal to
$\overline{ g( C_{X}\kappa (s_{1}), \kappa (s_{2})}$) and that $g$ is symmetric,
we obtain, from 
(\ref{relation-1}) and 
(\ref{relation-2}), that
\begin{equation}\label{relation-3}
\bar{X} g ((\tilde{A} - \tilde{A}^{*})(s_{1}), s_{2} ) = 0,\quad  s_{1},s_{2}\in {\mathcal O} (K).   
\end{equation}
We deduce  that if the third relation (\ref{def-harm-1}) 
(or equivalently, relation (\ref{relation-1})) 
is satisfied by $\tilde{A}$, then it is satisfied also by $\frac{1}{2} ( \tilde{A} + \tilde{A}^{*}).$
From $(B^{\flat})^{*} = (B^{*})^{\flat}$, for any  $B\in \Gamma (\mathrm{End}(K))$, we obtain the claim.
\end{remark}

\begin{proposition}\label{pot-cond}  A  CV-structure 
$(K,C, \UU , g, \kappa , h, D, \QQ )$
together with a Saito structure $(K,C, \UU , g, \nabla^{r}, \VV )$, 
which share the same data $(K\to M, C, \UU , g)$,  form a harmonic potential real Saito structure 
if and only if  the  $A(0)$-condition holds. 
\end{proposition}

{\bf Proof:} 
Let $(p^{*}(K), \nabla , P)$ be the $(TEP)(w)$-structure associated  to the Saito structure. Thus, 
$p^{*}(K)$ is a holomorphic bundle (the pull-back of the holomorphic bundle $K$ by the holomorphic map $p$). We denote by $\bar{\partial}$ its
d-bar operator.  The connection   $\nabla$ (viewed as a $(1,0)$-connection)
is given by
$$
\nabla_{X} (a) = \nabla^{r}_{X}(a) +\frac{1}{z} C_{X}(a),\ \nabla_{U}(a) = (\frac{1}{z} \UU - \VV +\frac{w}{2} )(a), 
$$
for any $X\in {\mathcal T}^{1,0}_{M}$ and $a\in p^{-1} \Gamma (K)$
(as usual, we identify data on $K$ with its natural lift to $p^{*}(K)$).

Similarly, let $(p^{*}(K), \wwh{\nabla}, \wwh{P})$ be the $(TEP)(w)$-structure which underlies 
the $(TERP)(w)$-structure associated to the CV-structure. Here $p^{*}(K)$ is considered with the holomorphic structure given by the kernel of the map
$$
\widehat{\nabla}_{\bar{X}}(a) = D_{\bar{X}}(a)+ z \kappa C_{X} \kappa (a),\quad a\in p^{-1}\Gamma (K),\ X\in {\mathcal T}^{1,0}_{M}.
$$ The $(1,0)$-connection determined by $\widehat{\nabla}$ is given by
$$
\widehat{\nabla}_{X}(a)= D_{X}(a) +\frac{1}{z} C_{X}(a),\ \widehat{\nabla}_{U}(a) = (\frac{1}{z} \UU - \QQ +\frac{w}{2} + z \kappa \UU \kappa )(a),
$$
for any $X\in {\mathcal T}^{1,0}_{M}$ and $a\in p^{-1} \Gamma (K).$
We remark that $ D^{(0,1)}_{\bar{X}}(a)= \bar{\partial}_{\bar{X}}(a) $ ($D$ is a Chern connection on $K$) and both vanish when $a\in p^{-1} {\mathcal O}(K)$. 
The pairings $P$ and $\widehat{P}$ coincide and are given  by
$$
P(a,b)(z, t)= \widehat{P}(a,b)(z,t)= z^{w} g(a,b),\quad a,b\in p^{-1} \Gamma (K),\ (z,t)\in \C \times M.
$$
By definition,  the $A(0)$-condition holds if and only if 
there is a formal automorphism 
$\Psi := \mathrm{id} + z\cdot A_{1} + z^{2}\cdot A_{2} + \cdots$, with  $A_{i} \in \Gamma ( \mathrm{End}(K))$, 
such that, for any  $a\in p^{-1}{\Gamma}(K)$ and $X\in {\mathcal T}_{M}$, 
\begin{align}
\nonumber\wwh{\nabla}_{X}  (\Psi (a)) = \Psi \nabla_{X}(a)\ \mathrm{mod} z{\mathcal O}[[z]],\\
\nonumber\wwh{\nabla}_{U}(\Psi (a)) = \Psi \nabla_{U}(a)\ \mathrm{mod} z {\mathcal O}[[z]],\\
\label{cond-det}\wwh{\nabla}_{\bar{X}} \psi (a ) = \psi \bar{\partial}_{\bar{X}} (a)\ \mathrm{mod} z^{2}{\mathcal O}[[z]]
\end{align}
and, for any $a, b\in p^{-1}{\Gamma}(K)$,
\begin{equation}\label{cond-ad}
z^{-w-1} P(\Psi (a), \Psi (b)) = z^{-w-1} P(a,b)\ \mathrm{mod}\ z {\mathcal O}[[z]]. 
\end{equation}
By identifying coefficients in $z$, we obtain that the first relation (\ref{cond-det}) is equivalent to 
$D_{X} = \nabla^{r}_{X} + [A_{1}, C_{X}].$ The second relation (\ref{cond-det}) is equivalent to
$\QQ = \VV + [\mathcal U , A_{1}]$ and the third to  
$$
\bar{\partial}_{\bar{X}} (A_{1}(a) ) - A_{1} \bar{\partial}_{\bar{X}} (a) + \kappa C_{X}\kappa (a) =0,
$$
i.e. to  $D^{(0,1)}_{\bar{X}} (A_{1}) = - \kappa C_{X}\kappa $. Finally, condition (\ref{cond-ad}) 
is equivalent to $g(A_{1}(a), b) = g (a, A_{1}(b)).$  
Our claim follows. A potential for the resulting harmonic potential real Saito bundle is  $- {A}_{1}^{\flat}$.
\hfill$\Box$

\medskip

In the semisimple case, the various formal equivalences considered in Definition \ref{var-equiv}  are the same. More precisely, the following holds.

\begin{corollary}\label{t2.18}
Consider a   CV-structure together with a Saito structure, with the same holomorphic vector bundle $K \to M$ over a simply connected manifold $M$,  
such that the associated $(TEP)$-structures are semisimple.  Then the $A(-1)$ condition implies all higher order $A(k)$-conditions (for any $0\leq k\leq \infty$). 
If it holds, the  CV-structure and the Saito structure  form
a harmonic potential real Saito structure.
\end{corollary}

\begin{proof}  
As already stated above, the $A(-1)$-condition means that the CV-structure and the Saito structure share the same data $(C, \mathcal U , g).$ In particular, the operator 
$\mathcal U$ (and also its eigenvalues) of the two $(TEP)$-structures is the same.
It is  semisimple (with different eigenvalues) and $M$ is simply connected. Therefore,  the hypothesis of Theorem \ref{t2.5} holds. The first claim follows from
Theorem \ref{t2.5}. The second claim follows from Proposition
\ref{pot-cond}.
\end{proof}

\section{Putting together Frobenius manifolds and $tt^*$ geometry}
\label{s3}
\setcounter{equation}{0}

\subsection{$F$-manifolds}\label{s3.1}

\noindent In this section we recall the relation between $(TEP)$-structures with the unfolding condition (see Definition \ref{t3.2} (b) below) and $F$-manifolds.  
We discuss a strong (if true)
conjecture from \cite{DH16} about $(TEP)$-structures with the
unfolding condition above a given $F$-manifold.

\begin{definition}\label{t3.1} \cite{HM99}\cite{He02}
A manifold $M$ with a commutative, associative
multiplication $\circ$ on the holomorphic
tangent bundle $TM$ and  a unit field $e\in\TT_M$ is an
{\sf $F$-manifold} if the multiplication satisfies
\begin{eqnarray}\label{3.1}
L_{X\circ Y}(\circ) =X\circ  L_Y(\circ)+Y\circ  L_X(\circ),\ X, Y\in {\mathcal T}_{M}.
\end{eqnarray}
A vector field $E\in\TT_M$ is called an {\sf Euler field of weight
$d\in\C$} if
\begin{eqnarray}\label{3.2}
L_E(\circ) =d\cdot \circ.
\end{eqnarray}
An Euler field of weight 1 is simply called an {\sf Euler field}.
\end{definition}

\bigskip

The $F$-manifolds we are interested in always come with a fixed Euler field. For this reason and to simplify terminology,
from now on by an $F$-manifold we always mean an $F$-manifold with a given Euler field.

\bigskip

\begin{definition}\label{t3.2}
(a) Let $(K, C,\UU)$ be a tuple where $K\to M$ is a
holomorphic vector bundle, $C$ is a Higgs field and
$\UU$ is a holomorphic endomorphism such that \eqref{2.7} is satisfied.
The tuple $(K, C, \UU )$ 
satisfies the {\sf unfolding condition} if
there is  a section $\zeta \in {\mathcal O}(K)$ (called primitive)
such that
\begin{equation}\label{def-I}
I= -C_\bullet\zeta:\TT_M\to\OO(K),\quad I(X):= - C_{X}(\zeta )
\end{equation}
is an isomorphism.

\medskip

(b) A $(TEP)$-structure on $\C\times M$ satisfies the {\sf
unfolding condition} if the induced data $(K\to M,C,\UU)$ (see Lemma
\ref{t2.2}) satisfies  it.
\end{definition}

\bigskip

The following lemma is \cite[Theorem 3.3]{HHP10}.

\begin{lemma}\label{t3.3}
A ($TEP)$-structure $(H \to \C \times M, \nabla , P)$  with the unfolding condition
induces the structure of an $F$-manifold  $(M,\circ ,e ,E)$. More precisely, let $(K\to M,C,\UU)$ be the restriction of $H$
to $\nmmm$, with Higgs field $C$ and endomorphism $\UU$ (see Lemma \ref{t2.2}). 
The multiplication $\circ$,  unit field $e$ and 
Euler field $E$  
are determined by: 
\begin{equation}\label{str}
C_{X\circ Y} (\zeta) = - C_{X}C_{Y}(\zeta),\  C_e (\zeta )=-\zeta,\ C_E(\zeta ) =- \UU (\zeta) ,
\end{equation}
for any $X, Y\in {\mathcal T}_{M}$.
\end{lemma}

\bigskip

The unfolding condition implies that the map 
$\TT_M \ni X \rightarrow - C_{X}\in \mathrm{End}(\OO(K))$ is 
an injection, and one can see that  relations (\ref{str}) are 
equivalent to the apparently stronger relations
\begin{equation}
C_{X\circ Y} = - C_XC_Y  ,\  C_e =-\id ,\ C_E 
=-\UU .
\end{equation}
In particular, $\circ$ and $E$ are independent of the choice of $\zeta .$
By means of the isomorphism (\ref{def-I}), the metric $g$ on $K$ induced by the $(TEP)$-structure 
(see Lemma \ref{t2.2}) induces a
multiplication invariant, non-degenerate bilinear form $g^{M}$ on $TM$.
As opposed to $\circ$ and $e$, $g^{M}$  is non-canonical (it depends on the choice of $\zeta$).
But its existence shows that all algebras $(T_tM,\circ_{t}, e_{t})$ are Gorenstein rings.

\begin{definition}\label{t3.4}
An $F$-manifold $(M,\circ ,e, E)$   is {\sf semisimple}
if the multiplication $\circ$  is everywhere semisimple (i.e.  any tangent space $T_{p}M$ is a direct sum of one dimensional algebras).  
An $F$-manifold is {\sf generically semisimple}
(in \cite{He02} it was called {\sf massive}) if it is semisimple
at generic points.
\end{definition}

By the local decomposition of $F$-manifolds \cite[Theorem
2.11]{He02}, a semisimple $F$-manifold is locally a product of
1-dimensional $F$-manifolds and the eigenvalues of $\UU=E\circ$
form locally Dubrovin's canonical coordinates.  
(Be aware that semisimplicity here does not require that these
eigenvalues are everywhere pairwise different.)

Special $F$-manifolds which are nowhere semisimple were studied in
\cite{DH15a}. The $(TEP)$-structures above them satisfy
themselves an unfolding condition, which was studied in
\cite{HM04}. It seems to be difficult to determine how many
$(TEP)(w)$-structures exist above other nowhere semisimple $F$-manifolds. 

But for irreducible germs of generically semisimple
$F$-manifolds, we have the following conjecture,
which we hope to address in the forthcoming paper
\cite{DH16}.

\begin{conjecture}\label{t3.5}
Let $((M,0),\circ,e,E)$ be an irreducible germ of a generically
semisimple $F$-manifold such that $(T_0M,\circ)$
is a Gorenstein ring. Then above it, there exists up to
isomorphism a unique $(TEP)(w)$-structure. Its only automorphisms
are $\pm\id$.
\end{conjecture}

\begin{remarks}\label{t3.5b}
(i) In various cases the existence is known: whenever the $F$-manifold
extends to a Frobenius manifold
(see Subsection \ref{s3.2}),  there is a Saito structure
and thus a $(TEP)(w)$-structure above the $F$-manifold.
But there are other cases (e.g. the 3-dimensional $F$-manifolds
from \cite[Theorem 5.30]{He02}), where existence
is unknown. Uniqueness is unknown in almost all cases.
That the only automorphisms are $\pm\id$ is true and elementary.

\medskip
(ii) In Conjecture \ref{t3.5} the irreducibility assumption 
of the germ is essential. 
Let $((M,0), \circ ,e, E)$ be a reducible germ of $F$-manifolds. 
It decomposes into irreducible germs (see \cite[Theorem 2.11]{He02}),
and this decomposition lifts 
to a formal decomposition of any $(TEP)$-structure above 
$((M,0), \circ ,e, E)$. 
(The forthcoming paper \cite{DH16}, which will address Conjecture \ref{t3.5},
will also contain this result). This is not, in general, a holomorphic decomposition.
Suppose now that $((M,0), \circ , e, E)$ is generically  semisimple (and reducible).
If Conjecture \ref{t3.5} is true 
then the $(TEP)(w)$-structures above
$((M,0), \circ ,e, E)$
differ (only) 
in the way in which the unique $(TEP)(w)$-structures above the irreducible
factors of $((M,0), \circ ,e, E)$ are put together (via formal isomorphisms).
This should be controlled by Stokes data.

\medskip
(iii) If Conjecture \ref{t3.5} is true, then for any  germ of a generically semisimple $F$-manifold,
the various conditions $A(-1)$, $A(0)$ and $A(\infty )$
from Definition \ref{var-equiv} are equivalent -  because of (ii).
If the germ is also irreducible, then all these conditions are equivalent to 
$A(hol )$ as well.
\end{remarks}

\subsection{Frobenius manifolds}\label{s3.2} 
We  consider only Frobenius manifolds with Euler fields.

\begin{definition}\label{t3.6}
(Dubrovin \cite{Du92}, see also \cite{Sa02} or \cite{He02})
A {\sf Frobenius manifold} 
is a tuple $(M,\circ,e,E,g,d)$ where $(M,\circ,e,E)$ is
an $F$-manifold, $g$ is a holomorphic, symmetric, non-degenerate
$\OO_M$-bilinear form with flat Levi-Civita connection
$\nnn^g$, $d\in\C$, and the following conditions hold:

\medskip   

(a)  the metric $g$ is multiplication invariant, i.e. $g(X\circ Y, Z) = g(X, Y\circ Z)$, for any $X, Y, Z\in {\mathcal T}_{M}.$

\medskip 

(b)   $\nnn^g e=0$ and  $L_E(g)=(2-d)\cdot g$.
\end{definition}

As proved in \cite{He02}, the $F$-manifold condition \eqref{3.1}
and $\nnn^g e=0$ imply the potentiality condition $\nnn^g(C)=0$.
The tangent bundle $TM$ of a Frobenius manifold underlies a Saito structure, with Higgs field $C_{X}(Y) = - X\circ Y$,
endomorphisms $\UU(X)  = X\circ E$, $\VV (X) = \nabla^{g}_{X}(E) -\frac{2-d}{2} X$, connection $\nabla^{r} = \nabla^{g}$  
and metric $g$.
This Saito structure satisfies the unfolding condition (take $\zeta = e$).
Conversely, any Saito structure with the unfolding condition and 
a certain global section defines a Frobenius manifold. 
This is expressed in the next theorem,  which  was  implicitly used in the
construction 
by K. Saito and M. Saito (1983) of Frobenius manifold structures
(under the name {\sf flat structures})
on the base space of a universal unfolding of an isolated
hypersurface singularity. But it was formalized only much later, first in \cite{Sa02}.

\begin{theorem}\label{t3.7}(\cite[Chapter VII, 3.6]{Sa02} , \cite[Theorem 5.12]{He03})
Let $(K\rightarrow M,C,\UU,g,\nnn^r,\VV)$ be a Saito structure with the unfolding condition. Let $\zeta\in\OO(K)$ be a global section such that
the map (\ref{def-I}) is an isomorphism,  $\nnn^r\zeta=0$ and
$\VV (\zeta )=\frac{d}{2}\zeta$ for some $d\in \C$.
By Lemma \ref{t3.3} $(M,\circ,e,E)$ is an $F$-manifold. Shift the Saito structure with $I^{-1}$
(the inverse of the isomorphism  $ TM \ni X \rightarrow I (X) = - C_{X}(\zeta )\in K$)
to the bundle $TM$. Then $g$ becomes $g^M$, and $\nnn^r$ becomes
the Levi-Civita connection of $g^M$. The tuple $(M,\circ,e,E,g^{M},d)$
is a Frobenius manifold.
\end{theorem}

\subsection{$CV$ structures on the tangent bundle}\label{s3.3}

Theorem 4.5 of \cite{He03} treated $CV$-structures $(TM,C,
\mathcal U, g, k, h, D, Q)$ on tangent bundles, with the unfolding condition. But two
facts were not included in its statement: namely, that $L_{e}(g)
=0$ implies  $D_{e} = L_{e}$ (where $e$ is the unit field of the induced $F$-manifold structure on $M$) and that the operator $\QQ$ is
uniquely determined by the other data of the CV-structure. These
facts lead to a stronger result. Theorem \ref{t3.8} below formulates
this stronger result. The first step in its proof is the next
lemma, which describes the endomorphism $\QQ$ 
of a CV-structure  (not
necessarily on a tangent bundle) with the unfolding condition.

\begin{lemma}\label{help} Let $(K \to M, C, \mathcal U, g, k, h, D, \QQ )$ be a CV-structure with the unfolding condition and $I: TM \rightarrow K$ the isomorphism
(\ref{def-I}) defined by a primitive section $\zeta .$ Then
\begin{equation}\label{expr-q-1}
\QQ= \left( D_{I^{-1}} ({\mathcal U})(\zeta )\right)^{skew}.
\end{equation}
Here $D_{I^{-1}}({\mathcal U})(\zeta )\in\Gamma ( \mathrm{End}(K))$ is
given by
$$
D_{I^{-1}} ({\mathcal U})(\zeta )(s):= D_{I^{-1}(s)} ({\mathcal U})(\zeta ),\quad s\in \OO  (K),
$$
and the superscript "$skew$" means its skew-symmetric part with
respect to $g$.
\end{lemma}

{\bf Proof:} Let $X_{0}\in {\mathcal T}_{M}$, such that $C_{X_{0}} (\zeta )= \QQ (\zeta )$. The relation
$$
D_{X}({\mathcal U})(\zeta ) - [C_{X}, \QQ](\zeta ) +C_{X}(\zeta )=0
$$
is equivalent to
$$
D_{X}({\mathcal U})(\zeta)- C_{X_{0}} C_{X}(\zeta ) + \QQ C_{X}
(\zeta )+ C_{X} (\zeta ) =0.
$$
Letting $s:= - C_{X} (\zeta )$, we obtain
\begin{equation}\label{q}
\QQ(s) = C_{X_{0}} (s) - s + D_{I^{-1}(s)}({\mathcal U}) (\zeta ),\
\forall s\in \Gamma (K).
\end{equation}
On the other hand, $\QQ$ is $g$-skew-symmetric. Replacing in the
relation
$$
g(\QQ(s_{1}), s_{2}) + g(s_{1},\QQ(s_{2})) =0
$$
the expression (\ref{q}) for $\QQ$, we obtain that $C_{X_{0}} -
\mathrm{id} = - (D_{I^{-1}}(\mathcal U )(\zeta ))^{sym}$. This
fact, together with (\ref{q}), implies (\ref{expr-q-1}).\hfill$\Box$

\bigskip

The next theorem is a stronger version 
 of \cite[Theorem 4.5]{He03}. As before, the superscript "$skew$" in 
 (\ref{expr-q}) below means the $g$-skew-symmetric part.

\begin{theorem}\label{t3.8}
Let $(TM,C,\UU,g,\kappa , h,D,\QQ)$ be a $CV$-structure on a tangent bundle $TM$, with the unfolding condition.

\medskip
(a)  $(M,\circ,e,E)$ with $C_{X} = - X\circ$,  $C_e = -\id$ and
$C_{E} =  - \UU$ is an $F$-manifold.

\medskip
(b) The operator $\QQ$ is given by
\begin{equation}\label{expr-q}
\QQ=(D_{E} - L_{E})^{skew}
\end{equation}
and satisfies
\begin{eqnarray}\label{3.3}
L_{\oooo{e}}(\QQ)=D_{\oooo{e}}(\QQ)=D_e(\QQ)=0.
\end{eqnarray}

\medskip
(c) We have the equivalences
\begin{eqnarray}\label{3.4}
&& D_e- L_e=0\quad\textup{on }C^\infty \textup{ tensors}\\
&\iff& D_ee=0\nonumber\\
&\iff& L_e(h)=0\iff  L_{\oooo{e}}(h)=0\nonumber\\
&\iff& L_e(\kappa)=0\iff L_{\oooo{e}}(\kappa)=0\nonumber\\
&\iff& L_e(g)=0\nonumber
\end{eqnarray}
and if these statements hold, then also
\begin{eqnarray}\label{3.5}
L_e(\QQ)=0.
\end{eqnarray}

\medskip (d) We fix $d\in \mathbb{R}.$ If
\begin{eqnarray}\label{3.7}
L_E(g)&=&(2-d)g
\end{eqnarray}
then
\begin{eqnarray}\label{3.8}
&\QQ =D_{E} - L_{E}- \frac{2-d}{2} \mathrm{id}\\
&L_{E-\bar{E}}(h)=0,\label{3.9}\\
&L_E(\QQ )=L_{\bar{E}}(\QQ)=D_E\QQ=D_{\bar{E}}\QQ
=[\UU,\kappa\UU\kappa].\label{3.10}
\end{eqnarray}
In particular,  the tensors $h$ and $\QQ$ are invariant under the
flow of $E-\bar{E}$.
\end{theorem}

{\bf Proof: } Part  (a) is proved  in \cite{He03}. 
It follows from Lemma  \ref{t3.3}
combined with the fact that any CV-structure arises from a $(TERP)$-structure  (in particular, from a $(TEP)$-structure).
To prove
(\ref{expr-q}), we consider
in Lemma \ref{help}  $\zeta =e$,
hence $I(X) =  X$ for any $X\in TM.$ Using that ${\mathcal U}(X) =
E\circ X$, we obtain
\begin{align}
2g(\QQ (X), Y) \nonumber&=  g (D_{X}({\mathcal U})(e), Y) - g (D_{Y}({\mathcal U})(e), X)\\
\nonumber&= g (D_{X}(E) - E\circ D_{X}(e) , Y)\\
\label{rel1}& -g (D_{Y}(E) - E\circ D_{Y}(e) , X).
\end{align}
On the other hand, from $D(C) =0$,
$$
0 = (D(C))(X, E)(e) = D_{X}(C_{E})(e) - D_{E}(C_{X})(e) - C_{[X,
E]}(e).
$$
This  gives, using $C_{X}(Y) = - X\circ Y$ and
$$
D_{X}(C_{Y}) (Z) = D_{X} ( C_{Y}(Z) ) - C_{Y} D_{X}(Z)
$$
the following relation:
\begin{equation}\label{rel2}
D_{X}(E) - D_{X}(e) \circ E = (D_{E}- L_{E})(X) - D_{E}(e) \circ
X.
\end{equation}
Combining (\ref{rel1}) with (\ref{rel2}) we obtain
$$
2g (\QQ (X), Y) = g ( (D_{E} - L_{E})(X), Y) - g(X, (D_{E}-L_{E})(Y))
$$ which is (\ref{expr-q}). The equalities (\ref{3.3}) were proved
in \cite{He03}. Part (b) follows.

In part (c), only the implication $L_e(g)\Longrightarrow
D_e- L_e=0$ is not contained in \cite[Theorem 4.5]{He02}. In
order to prove it, we make the following computation: for any
$X\in\TT_M$,
\begin{eqnarray}
[C_X,D_e- L_e]&=&-[X\circ,D_e- L_e] =D_e(X\circ)-L_e(X\circ)
\nonumber\\
&=& (D_e(X\circ)-D_X(e\circ))- L_e(X)\circ
\nonumber\\
&=&[e,X]\circ -[e,X]\circ=0,\label{3.11}
\end{eqnarray}
where in the second line we used 
$D_{X} (e\circ )=0$  and $L_{e}(\circ )=0$  and in last
line we used $(DC)(e,X)=0$. Therefore, there is a 
vector field $\xi\in\TT^{1,0}_M$, such that $D_e- L_e=\xi\circ$.
Now $L_e(g)=0$ and \cite[(4.27)]{He03}, which is
\begin{eqnarray}\label{3.12}
L_X(g)(Y,Z)=g((D_X- L_X)(Y),Z)+g(Y,(D_X- L_X)(Z)),
\end{eqnarray}
show together
\begin{eqnarray}
0&=& L_e(g)(Y,Z) \nonumber\\
&=& g((D_e- L_e)(Y,Z)+g(Y,(D_e- L_e)(Z))\nonumber\\
&=& g(\xi\circ Y,Z)+g(Y,\xi\circ Z)=2g(\xi\circ Y,Z)\nonumber .
\end{eqnarray}
We obtain that $\xi=0$ and $D_e- L_e=0$.

(d) Relation \eqref{3.8} follows from \eqref{expr-q}, $L_E(g)=(2-d)g$ and
\eqref{3.12} for $X=E$.
The rest of part (d) was proved in  \cite[Theorem 4.5]{He03} (in fact
the part behind the proof of (4.22) in \cite{He03} is not needed
anymore, due to (\ref{3.8})). \hfill$\Box$

\bigskip

In view of the above theorem, the notion of a 
$CDV$-structure from
\cite[Definition 4.6]{He03} can be defined in the following simplified way.

\begin{definition}\label{t3.9}
A {\sf CDV-structure} on a manifold $M$ is a $CV$-structure
$(TM\to M,C,\UU,g,\kappa,h,D,\QQ)$, together with a
Frobenius manifold structure $(M,\circ,e,E,g)$, such
that $C_X=-X\circ$, $  \UU (X)=E\circ X$, for any $X\in TM$
and the constant $d\in \mathbb{C}$ from $L_{E}(g) = (2-d)g$ is real.
 It is a {\sf  $CDV\oplus$-structure} if $h$ is positive definite.
\end{definition}

\begin{remarks}\label{t3.10} (i) In \cite[Definition 4.6]{He03}, it was required, in addition 
to the conditions from Definition \ref{t3.9}, that
the operator $\QQ$ of a CDV-structure is given by (\ref{3.8}) and the 
equivalent statements in \eqref{3.4} hold. Now these requirements
follow from Theorem \ref{t3.8}, part (d) respectively, part (c) 
(because  
$L_E(g)=(2-d)g$, respectively $L_e(g)=0$  on a Frobenius manifold).

\medskip

(ii) With the stronger version of \cite[Theorem 4.5]{He03}
and the new definition of a $CDV$-structure, Theorem
5.15 in \cite{He03} becomes trivial.

\medskip
(iii) Lin analyzed the definition of a $CDV$-structure
and reformulated it in \cite[Theorem 2.1]{Li11} as 
a Frobenius manifold $(M,\circ,e,E,g)$ with a real structure
$\kappa$ such that the following small system of additional
conditions is satisfied: $\QQ$ defined 
by  (\ref{3.8})   is self-adjoint with respect to $h$, and, moreover,
\eqref{2.27} and \eqref{2.28} hold.

\medskip

(iv) A $CDV$-structure on $M$ gives two $(TEP)(w)$-structures above $M$,
one from the Frobenius manifold (more precisely, from the associated Saito structure on $TM$), the other from the $CV$-structure. One can again use the 
various alternatives $A(-1),A(0),...,A(\infty)$ and $A(hol)$ from Definition 
\ref{var-equiv} to formulate a priori stronger compatibilities between the
Frobenius manifold and the $CV$-structure. A $CDV$-structure 
for which the Saito and CV-structure form a harmonic potential real Saito structure  
(or the $A(0)$-condition holds) is considered in \cite[Chapter 1]{Sa08} and is called
there a {\sf harmonic Frobenius manifold}. A $CDV$-structure where
$A(\infty )$  holds is called a {\sf  potential $CDV$-structure} in
\cite{LS12}.  A CDV-structure for which the $A(hol)$-condition holds is called a  
{\sf  strongly potential $CDV$-structure} in \cite{LS12}.
See Corollary \ref{t2.18} and 
Remark \ref{t3.5b} (iii) for the comparison of $A(-1),A(0),A(\infty)$ and 
$A(hol)$ in the semisimple and generically semisimple cases.

\medskip

(v) Lin studied in \cite[Theorem 2.4]{Li11} the case of a
$CDV$-structure with underlying semisimple $F$-manifold, where the
Frobenius manifold is arbitrary (i.e. its $(TEP)(w)$-structure has
an arbitrary Stokes structure), but the $CV$-structure
induces a $(TEP)(w)$-structure with trivial Stokes structure. Then
$\QQ=0$  and formulas for the
$CV$-structure are very explicit. The CV-structure  is obtained by pushing
an abstract $CV$-structure, which is a sum of rank 1 structures,
with a section $\zeta$ as in Theorem \ref{t3.7} to the tangent
bundle. These $CDV$-structures are also studied in \cite{LS12}.
\end{remarks}

\section{Canonical data from the endomorphism bundle}\label{s4}
\setcounter{equation}{0}

Let $(K\to M, C,\UU,g,\kappa,h,D,\QQ)$ be a
$CV$-structure with the unfolding condition.  
Owing to the unfolding condition,
the map
$$-C_\bullet: TM\to \End(K),\ X \rightarrow - C_{X}$$
is injective, so it restricts to an isomorphism 
\begin{eqnarray}\label{4.1}
-C_\bullet: TM\to F 
\end{eqnarray} from $TM$ to the bundle
$$
F:= \{ C_{X}\, |\, X\in TM\} \subset \mathrm{End}(K).
$$
On the bundle $F$ we define the data 
$(C^{F}, h^{F}, g^{F}, \QQ^{F})$, formed by a holomorphic Higgs field $C^{F}$, a hermitian form $h^{F}$ (which will be assumed to be non-degenerate), 
a holomorphic  complex bilinear form $g^{F}$ and a $C^{\infty}$-endomorphism $\QQ^{F}$.
The first three pieces of the data are defined as
$$
C^{F}_{X} (C_{Y}) := C_{X}C_{Y},\quad  h^{F}:= h^{end}\vert_{F},\ g^{F}:= g^{end}\vert_{F}.
$$
Here $h^{end}$ is the natural hermitian metric  on $\mathrm{End}(K)$ induced   
by $h$. It is obtained as follows:  $h$ induces a ($\mathbb{C}$-anti-linear) isomorphism $X \rightarrow h(\cdot , X)$ between $K$ and $K^{*}$ and 
a hermitian metric $h^{*}$ on $K^{*}$, defined by
$h^{*}( h(\cdot , X), h(\cdot , Y)) := h(Y, X)$, for any $X ,Y.$     
The metric $h^{end}= h^{*}\otimes h$ is the product metric on $\mathrm{End}(K) = K^{*}\otimes K$ and is given by
\begin{equation}\label{metric-end}
h^{end}(A, B) = \sum h (A(e_{i}), B(e^{*}_{i}))
\end{equation}
where $\{ e_{i}\}$ is a basis of $K$ and  $\{ e_{i}^{*} \}$ is the $h$-dual basis,
defined by $h(e_{i}, e_{j}^{*}) = \delta_{ij}$, for any $i, j$. 
The complex bilinear form  $g^{end}$  on $\mathrm{End}(K)$ is induced by $g$ and is defined
in the same way as $h^{end}$ (with $h$ replaced by $g$).

We assume, along the entire section, that $h^{F}$ is non-degenerate.
Then $\mathrm{End}(K) = F \oplus F^{\perp}$ is a direct sum decomposition,
where $F^{\perp}$ is the $h^{end}$-orthogonal complement of $F$. We denote by $\mathrm{pr} ^{F}$, $\mathrm{pr}^{F^{\perp}}$
the projections from $\mathrm{End}(K)$ onto its components $F$ and $F^{\perp}$. Then
$$
\QQ^F : F \rightarrow F,\quad \QQ^{F}:= \mathrm{pr}^{F} \circ \QQ^{end}\vert_{F}
$$
where 
$$
\QQ^{end} : \mathrm{End}(K) \to \mathrm{End}(K),\  \QQ^{end} (A):= [ \QQ , A],\ A\in \mathrm{End} (K).
$$

\medskip 

By Theorem \ref{t2.12} the $CV$-structure
$(K\to M, C,\UU,g,\kappa,h,D,\QQ)$ induces a 
$(TEP)$-structure with the unfolding condition.
Because of Lemma \ref{t3.3}, $M$ inherits a multiplication $\circ$ 
and a vector field $E$ which make it an $F$-manifold. Using the isomorphism  
\eqref{4.1}, the data $(C^{F}, h^{F}, g^{F}, \QQ^{F})$ induces 
a holomorphic  Higgs field $C^{M}$, a hermitian (non-degenerate) metric $h^{M}$, a holomorphic 
complex bilinear form $g^{M}$ and a $C^{\infty}$-endomorphism
$\QQ^{M}$ on $TM.$ We remark that $C^{M}_{X}(Y) = - X\circ Y$ for any $X$ and $Y$.
Since $\mathcal U = - C_{E}$, it is a section of $F$ and corresponds to $E$ by means of (\ref{4.1}).

\begin{definition}\label{t4.1}
Let $(K\to M,C,\UU,g,\kappa,h,D,\QQ)$ be a $CV$-structure with the
unfolding condition. The system  $(\circ , E,  h^{M}, g^{M}, \QQ^{M})$
is  called the {\sf canonical data} on $M$.  The hermitian form $h^M$ is
called the {\sf (hermitian) $CV$-metric} of the $F$-manifold $(M,
\circ ,e, E).$  
\end{definition}

The aim of this section is to study the properties of the canonical data.
To achieve this,  
in the following two lemmas  we study  the properties  of the system $(C^{F}, h^{F}, g^{F}, \QQ^{F})$.
The bundle $F$ will be considered as a subbundle of the hermitian vector bundle $(\mathrm{End}(K), h^{end}).$ 
Good references  for the theory of  holomorphic subbundles in holomorphic  hermitian vector bundles
(Chern connection of the subbundle,  its curvature, the second fundamental form)
are e.g. \cite[Chapter 1]{Ko87} and \cite[Chapter 0.5]{GH}. 

The Chern connection of $h^{end}$ is the natural extension of the Chern connection $D$  of $h$ to $\mathrm{End}(K)$ and will also be denoted by $D$.
The Chern connection $D^{F}$ of $h^{F}$ is
given by $D^{F}_{X}(C_{Y}) = \mathrm{pr}^{F}D_{X}(C_{Y})$ and the second fundamental form $A^{F}\in \Omega^{1} (\mathrm{Hom} (F, F^{\perp}))$ 
of the subbundle $F \subset \mathrm{End}(K)$ 
by
\begin{eqnarray}\label{added}
A^{F}_{X}(C_{Z})=  D_{X}(C_{Z}) - \mathrm{pr}^{F}
D_{X}(C_{Z}) . 
\end{eqnarray}

\begin{lemma}\label{t4.4} The following statements hold:

\medskip

(a) The second fundamental form $A^F$ is
symmetric in the following sense:
\begin{eqnarray}\label{4.6}
A^F_X(C_Y)=A^F_Y(C_X),\quad  X,Y\in \TT_M.
\end{eqnarray}

\medskip

(b) For any $X, Y, Z\in {\mathcal T}_{M}$,
\begin{equation}\label{changed}
(D^{F}(C^{F}))(X, Y)(C_{Z}) = \mathrm{pr}^{F} ( C_{Y}
A^{F}_{Z}(C_{X}) - C_{X} A^{F}_{Z}(C_{Y})).
\end{equation}
(c) The connection  $D^{F}$ preserves $g^{F}.$ 
\end{lemma}

{\bf Proof:} 
Part (a) follows by projecting the relation
\begin{equation}\label{pot}
(D(C))(X, Y) = D_{X}(C_{Y}) - D_{Y} (C_{X}) - C_{[X, Y]}=0
\end{equation}
onto $F^{\perp}$.

We now prove part (b). Using $D^{F}_{X} (C_{Y})=\mathrm{pr}^{F}
D_{X} (C_{Y})$ and the
definition of the Higgs field $C^{F}$, we obtain:
\begin{align}
\nonumber&(D^{F}(C^{F}))(X, Y) (C_{Z}) =
(D^{F}_{X} (C^{F}_{Y}) - D^{F}_{Y} (C^{F}_{X}) - C^{F}_{[X,Y]} )(C_{Z})\\
\nonumber&= D^{F}_{X} (C^{F}_{Y}C_{Z}) - C^{F}_{Y}D^{F}_{X}(C_{Z})\\
\nonumber&  - D^{F}_{Y} (C^{F}_{X}C_{Z}) + C^{F}_{X}D^{F}_{Y}(C_{Z})- C_{[X,Y]} C_{Z}\\
\nonumber&= \mathrm{pr}^{F}D_{X} (C_{Y}C_{Z}) - C_{Y}\mathrm{pr}^{F}D_{X}(C_{Z}) \\
\nonumber&-  \mathrm{pr}^{F}D_{Y} (C_{X}C_{Z}) + C_{X}\mathrm{pr}^{F}D_{Y}(C_{Z}) - C_{[X, Y]} C_{Z}\\
\nonumber&= \mathrm{pr}^{F} ( D_{X}(C_{Y}) C_{Z} +
C_{Y}D_{X}(C_{Z})) - C_{Y}
\mathrm{pr}^{F} D_{X}(C_{Z})\\
\label{calcul}& -  \mathrm{pr}^{F} ( D_{Y}(C_{X}) C_{Z} +
C_{X}D_{Y}(C_{Z})) + C_{X} \mathrm{pr}^{F} D_{Y}(C_{Z})- C_{[X,
Y]}C_{Z}.
\end{align}
On the other hand, composing (\ref{pot}) with $C_{Z}$,
$$
D_{X}(C_{Y}) C_{Z} - D_{Y}(C_{X}) C_{Z} - C_{[X, Y]}  C_{Z}=0.
$$
Projecting this relation to $F$ we obtain
$$
\mathrm{pr}^{F} ( D_{X} (C_{Y})C_{Z} -D_{Y}(C_{X})C_{Z}) = C_{[X,
Y]}C_{Z}.
$$
Using this relation, (\ref{calcul}) becomes
\begin{align}
(D^{F}(C^{F}))(X, Y) (C_{Z}) \nonumber&= \mathrm{pr}^{F}
(C_{Y}D_{X}(C_{Z}))
- C_{Y} \mathrm{pr}^{F}(D_{X}(C_{Z}))\\
\label{new}&  - \mathrm{pr}^{F}(C_{X} D_{Y}(C_{Z}))+ C_{X}
\mathrm{pr}^{F} (D_{Y}(C_{Z})).
\end{align}
On the other hand, using (\ref{added}) ,  we can write
\begin{align}
\nonumber&\mathrm{pr}^{F} (C_{Y}D_{X}(C_{Z})) - C_{Y}
\mathrm{pr}^{F}(D_{X}C_{Z})\\
\nonumber&= \mathrm{pr}^{F} (C_{Y}\mathrm{pr}^{F}D_{X}(C_{Z})) +
\mathrm{pr}^{F} (C_{Y}A^{F}_{X}(C_{Z})) - C_{Y}
\mathrm{pr}^{F}(D_{X}
(C_{Z}))\\
\label{r0}&= \mathrm{pr}^{F} (C_{Y} A^{F}_{X}(C_{Z})),
\end{align}
because $C_{Y}\mathrm{pr}^{F} D_{X}(C_{Z})$ is a section of $F$
(so, it coincides with its projection to $F$). In a similar way,
we obtain
\begin{equation}\label{r1}
\mathrm{pr}^{F} (C_{X}D_{Y}(C_{Z})) - C_{X} \mathrm{pr}^{F}
(D_{Y}C_{Z})= \mathrm{pr}^{F} (C_{X} A^{F}_{Y}(C_{Z})).
\end{equation}
From (\ref{new}), (\ref{r0}) and (\ref{r1}),
$$
(D^{F}(C^{F}))(X, Y) (C_{Z}) = \mathrm{pr}^{F} (C_{Y}
A^{F}_{X}(C_{Z})) - C_{X} A^{F}_{Y}(C_{Z})).
$$
Using $A^{F}_{X}(C_{Z})= A^{F}_{Z}(C_{X})$ and $A^{F}_{Y}(C_{Z})=
A^{F}_{Z}(C_{Y})$ we obtain (\ref{changed}). Part (b) is proved.
For part (c), we use $D( g^{end})=0$ (because $D(g)=0$) which implies
$$
D^{F}_{Z} (g^{F}) (C_{X}, C_{Y}) = D_{Z} (g^{end}) (C_{X}, C_{Y}) = 0
$$
as needed.
\hfill $\Box$

\bigskip

The next lemma studies the properties of 
$(C^{F}, h^{F}, g^{F}, \QQ^{F})$  which involve the endomorphism $\QQ^{F}.$

\begin{lemma}\label{t4.5}
(a) For any $X, Y\in {\mathcal T}_{M}$,  
\begin{eqnarray}
\nonumber &D^{F}_{X}({\mathcal U}) + \QQ^{F} (C_{X}) + C_{X} =0,\\
&D^{F}_{X} (\QQ^{F}) (C_{Y}) = D^{F}_{Y} (\QQ^{F}) (C_{X}).\label{suppl}
\end{eqnarray}

(b) For any $X, Y\in {\mathcal T}_{M}$,
\begin{eqnarray}
\nonumber&h^{F} ( \QQ^{F}(C_{X}), C_{Y}) = h^{F} (C_{X},
\QQ^{F}(C_{Y}))\\
\nonumber&g^{F} (\QQ^{F} (C_{X}), C_{Y})  + g^{F} (C_{X},\QQ^{F} (C_{Y}))\\
&= - g^{end} (\mathrm{pr}^{F^{\perp}} \QQ^{end}(C_{X}), C_{Y}) -
g^{end} (\mathrm{pr}^{F^{\perp}}
\QQ^{end}(C_{Y}), C_{X}).\label{suppl-1}
\end{eqnarray}
\end{lemma}

{\bf Proof:} The first relation (\ref{suppl}) follows by
projecting (\ref{2.30}) to $F$. We now prove the second relation
(\ref{suppl}). 
We shall use repeatedly the relation
\begin{align}
\nonumber& X h^{end}(A, B) = h^{end} (D_{X}(A), B) + h^{end}(A, D_{\bar{X}} (B)) \\
\label{deriv-chern}& = h^{end}(D_{X}(A), B),\ X\in {\mathcal T}^{1,0}_{M}.
\end{align}
for any $A\in \Gamma (\mathrm{End}(K))$ and $B\in {\mathcal O}(\mathrm{End}(K))$
(because $D_{\bar{X}} (B) = \bar{\partial}_{\bar{X}} (B) =0$, when $B$ is holomorphic).
Consider now  any $X, Y, Z\in {\mathcal T}_{M}$. Then
\begin{align*}
&h^{end}( D^{F}_{X}(\QQ^{F}) (C_{Y}), C_{Z}) 
= h^{end} (D^{F}_{X}(\QQ^{F}(C_{Y})) - \QQ^{F} D^{F}_{X} (C_{Y}), C_{Z})\\
&= h^{end} (D_{X}(\mathrm{pr}^{F}[\QQ, C_{Y}]) , C_{Z}) - 
h^{end}([ \QQ ,D^{F}_{X}(C_{Y})], C_{Z})\\
\nonumber&= X h^{end}([\QQ, C_{Y}], C_{Z}) 
-h^{end}([\QQ ,D^{F}_{X}(C_{Y})], C_{Z})
\end{align*}
where in the last equality we used
(\ref{deriv-chern}), with  $A:= \mathrm{pr}^{F} [ \QQ , C_{Y}]$ and $B:= C_{Z}.$
Therefore, 
\begin{align}
\nonumber&h^{end}( D^{F}_{X}(\QQ^{F}) (C_{Y}), C_{Z}) 
= X h^{end}([\QQ, C_{Y}], C_{Z}) 
-h^{end}([\QQ ,D^{F}_{X}(C_{Y})], C_{Z})\\
\nonumber&= h^{end}([D_{X}(\QQ), C_{Y}], C_{Z})+ h^{end}( [ \QQ, D_{X}(C_{Y})], C_{Z})
-h^{end}([\QQ ,D^{F}_{X}(C_{Y})], C_{Z})\\
\nonumber&= h^{end}([D_{X}(\QQ), C_{Y}], C_{Z})+ h^{end}(D_{X}(C_{Y}),[\QQ,
C_{Z}]) -
h^{end}(D^{F}_{X}(C_{Y}), [\QQ, C_{Z}])\\
\label{skew-part}& = 
h^{end}([D_{X}(\QQ), C_{Y}], C_{Z})+ h^{end}(A^{F}_{X}(C_{Y}),[\QQ,
C_{Z}])
\end{align}
where in the second line we used  again (\ref{deriv-chern}), with $A:= [ \QQ , C_{Y}]$ and
$B:= C_{Z}$, and afterwards the relation
\begin{equation}\label{prop-h}
h^{end}([A, B], C) = - h^{end}(A, [B^{\flat}, C]),\quad A, B, C\in
\mathrm{End}(K)
\end{equation}
and that $\QQ$ is $h$-hermitian. 
Using that $A^{F}_{X}(C_{Y}) = A^{F}_{Y}( C_{X})$, 
$D_{X}(\QQ) =- [C_{X}, \kappa \mathcal U \kappa ]$
and (\ref{prop-h}) again, we obtain
\begin{align*}
&h^{end}( D^{F}_{X}(\QQ^{F}) (C_{Y})- D^{F}_{Y}(\QQ^{F}) (C_{X}),
C_{Z}) \\
&= h^{end} ([D_{X}( \QQ ), C_{Y}], C_{Z} ) - h^{end} ( [D_{Y}(\QQ ), C_{X}], C_{Z})\\
&=- h^{end} ( [C_{X},\kappa \mathcal U \kappa ] , C_{Y}]  -
[C_{Y},\kappa \mathcal U \kappa ] , C_{X}], C_{Z})
\end{align*}
which is zero (from the Jacobi identity for $C_{X}$, $C_{Y}$,
$\kappa {\mathcal U}\kappa$ and $[C_{X}, C_{Y}]=0$). The second
relation (\ref{suppl}) is proved. Part (a) follows.

We now prove part (b). 
First, we show that  $\QQ^{F}$ is $h^{F}$-hermitian: using 
(\ref{prop-h}) and that $\QQ$ is $h$-hermitian, we can write 
\begin{align}
\nonumber &  h^{F} (\QQ^{F}(C_{X}), C_{Y}) = - h^{end}( [C_{X}, \QQ ],
C_{Y}) =h^{end}(C_{X}, [\QQ , C_{Y}])\\
\label{skew-h}&= 
h^{end}(C_{X},\QQ^{end}(C_{Y}))  =  h^{F}(C_{X}, \QQ^{F}(C_{Y})).
\end{align}
The second relation (\ref{suppl-1}) follows from a similar computation, 
which uses that  $\QQ$ is $g$-skew-symmetric and
$$
g^{end}([A, B], C) = - g^{end}(A, [B^{*}, C]),\quad A, B, C\in
\mathrm{End}(K).
$$
\hfill $\Box$

\bigskip

We now interpret Lemmas \ref{t4.4} and \ref{t4.5} on $TM.$ 
Because of Lemma \ref{t4.4} (b), the potentiality condition $D^{M}(C^{M}) =0$ 
(where $D^{M}$ is the Chern connection of $h^{M}$)
does not hold. Because of Lemma \ref{t4.5} (b), $\QQ^{M}$ is not, in general, $g^{M}$ 
skew-symmetric. But the other statements from these two lemmas translate nicely in terms of the canonical data. 
The following proposition collects various properties
of the canonical data.

\begin{proposition}
\label{canonical-TM} 
Let $(\circ , E, h^{M}, g^{M}, \QQ^{M})$ be the canonical data on $M$.

\medskip

(a) The Chern connection $D^{M}$ of $h^{M}$  preserves $g^{M}$, and, for any $X, Y\in {\mathcal
T}_{M}$,
$$
D^{M}_{X} (Q^{M}) (Y) = D^{M}_{Y} (Q^{M}) (X),\quad  D^{M}_{X}(E)
= Q^{M} (X) + X.
$$
The operator $Q^{M}$ is $h^{M}$-hermitian. In general, it is not
$g^{M}$-skew-symmetric.
\medskip

(b) The CV-metric $h^{M}$ is invariant under the flows of $e$, $\bar{e}$
and $E - \bar{E}.$

\end{proposition}

{\bf Proof: } Part (a)   follows from Lemmas \ref{t4.4} and \ref{t4.5}. We now prove part (b).
For any $X, Y\in {\mathcal T}_{M}$, 
\begin{align*}
L_{e} (h^{M}) (X, Y) & = e h^{M} (X, Y) - h^{M} ( [e, X], Y) - h^{M} ( X, [\bar{e}, Y])\\
&= e h^{M} (X, Y) - h^{M} ( [e, X], Y) \\
&= e h^{end} (C_{X}, C_{Y}) - h^{end}(C_{[e,X]} , C_{Y})\\
&= h^{end} (D_{e}(C_{X})- C_{[e, X]} , C_{Y}) =0,
\end{align*}
where we used $[\bar{e}, Y ]=0$ and 
$D_{e} (C_{X}) =C_{[e, X]}$ (the latter follows from $(D(C))(e, X)=0$ and $C_{e} = -\sf id $). 
In a similar way we obtain
$$
L_{\bar{e}}( h^{M})(X, Y) = h^{end} (C_{X}, D_{e}(C_{Y})  - C_{[e,Y]}) =0.
$$
It remains to prove that $L_{E- \bar{E}}(h^{M})=0.$ Using that
$$
D_{E}(C_{X}) - C_{[E, X]} = D_{X} (C_{E}) = - D_{X}(\UU )
$$
we obtain, by similar computations,
\begin{align*}
L_{E} (h^{M}) (X, Y)& = - h^{end} (D_{X}(\UU ), C_{Y}) =- h^{F} (D_{X}^{F}(\UU ), C_{Y})\\
L_{\bar{E}} (h^{M}) (X, Y)& = - h^{end} (C_{X}, D_{Y} (\UU )) =  - h^{F} (C_{X}, D^{F}_{Y} ( \UU ) ),
\end{align*}
for any $X, Y\in {\mathcal T}_{M}.$ 
We conclude from the first relation (\ref{suppl}) and the first relation (\ref{suppl-1}). 
\hfill $\Box$

\bigskip

For the proof of the next theorem we need to recall the relation between the curvatures of a  hermitian holomorphic vector bundle and a holomorphic subbundle, in terms of the second
fundamental form. Recall that   if $F\subset V$ is a holomorphic subbundle 
of a hermitian holomorphic vector bundle $(V, h^{V})$, such that $h^{V}$
restricts to a (non-degenerate) hermitian metric $h^F$ on $F$, then the curvatures $R^{V}$ and $R^{F}$ of the Chern connections $D^{V}$ and $D^{F}$ 
of $h^{V}$ and $h^{F}$ are related by (see e.g. \cite[Ch. I (6.12)]{Ko87}):
\begin{eqnarray}\label{4.3}
R^F=\mathrm{pr}^{F} \circ R^V+(A^F)^{\flat}\land A^F,
\end{eqnarray}
where $(A^F)^{\flat}\in \Omega^{0,1}(\mathrm{Hom} (F^{\perp}, F)) $ is the hermitian dual of $A^F\in \Omega^{1,0} (\mathrm{Hom}(F, F^{\perp}))$, given by 
$$h^{V}(A^F_{X} s,\tilde{s})=h^{V}(s,(A^F)_{\bar{X}}^{\flat}\tilde{s})\quad s\in \Gamma (F),\ 
\tilde{s}\in \Gamma (F^\perp),\ X\in {\mathcal T}^{1,0}_{M}.
$$
Consider now the special case when  $V:= \mathrm{End}(K)$ 
is an endomorphism bundle and $h^{V} := h^{end}$ is the natural extension of a hermitian metric $h$ on $K$ to
$V$. Suppose that 
the restriction $h^{F}$ of $h^{V}$ to $F$  is non-degenerate. The curvature $R^{end}$ of $h^{end}$ is naturally related to the curvature $R^{D}$ of $h$ by 
\begin{equation}\label{hom-curv}
R^{end}(X,\bar{Y})(A)= [ R^{D}(X,\bar{Y}), A],\quad X, Y\in {\mathcal T}^{1,0}_{M}.
\end{equation}
From  (\ref{prop-h}),  (\ref{4.3}) and (\ref{hom-curv}),  the curvature $R^F$ of $h^{F}$ is given by
\begin{align}
\nonumber  h^F(R^F(X,\oooo X)(B),B)&=h^{end}(R^{D}(X,\oooo X),[B,B^\flat ])\\
\label{4.4}&  -h^{end}(A^F_X(B),A^F_X(B)),
\end{align}
for any  $X\in {\mathcal T}_{M}^{1,0}$ and $B\in \Gamma (F).$  

\medskip 

Our main result from this section is the following.

\begin{theorem}\label{t4.6}
Let $(K\to M,C,\UU,g,\kappa,h,D,\QQ)$
be a $CV\oplus$-structure with  the unfolding condition.
Then the hermitian forms $h^{end},h^F$ and $h^M$ are automatically positive definite.
Let $(\circ , e)$ be the multiplication and the unit field of the  induced $F$-manifold structure on $M$.

\medskip

(a) The holomorphic sectional curvature  
$R^{sect}$ of $h^{M}$ is non-positive,
\begin{eqnarray}\label{4.7}
R^{sect}(X):=\frac{h^M(R^M(X,\bar{X})X,X)}{h^M(X,X)^{2}}
\leq 0,\quad X\in TM\setminus \{ 0 \}
\end{eqnarray}
and vanishes in the direction of  $e$:
\begin{eqnarray}\label{4.8}
R^{sect}(e)=0.
\end{eqnarray}

(b) Let $t\in M$ be such that $(T_{t}M, \circ_{t} ,e_{t})$ is an irreducible algebra. The decomposition 
$T_{t}M = \mathbb{C} e_{t} \oplus I^{max} (T_{t}M)$ of $T_{t}M$ into 
the line $\mathbb{C}e_{t}$ and the maximal ideal 
$I^{max}(T_{t}M)$, extends, in a neighbourhood $V$ of $t$, to a decomposition 
\begin{equation}\label{dec-hyp}
T_{\tilde{t}} M =\mathbb{C} e_{\tilde{t}}  + U_{\tilde{t}},\quad  \mathbb{C} e_{\tilde{t}} \cap U_{\tilde{t}} = \{  0\},
\end{equation}
with the following properties:

(i) for any $\tilde{t}\in V$, $U_{\tilde{t}}$ is a cone of  $T_{\tilde{t}}M$ (i.e.  $r\cdot U_{\tilde{t}} \subset U_{\tilde{t}}$,
 for any $r\geq 0$) and 
$I^{max} (T_{t}M)$ is included  in $U_{t}$;\\

(ii) there is $k_{0}> 0$ such that
\begin{equation}\label{h1}
R^{sect} (X) < - k_{0},\quad X\in U_{\tilde{t}}\setminus \{ 0 \},\  \tilde{t}\in V.
 \end{equation}

(iii) for any $\lambda_{0} >0$, there is $k_{1} >0$ such that
 \begin{equation}\label{h2}
 R^{sect} (X) < -k_{1}, 
 \end{equation}
 for any $X\in T_{\tilde{t}}V$, with $X - e_{\tilde{t}}\in U_{\tilde{t}}\setminus \{ 0\}$ and
 $h^{M} (X- e_{\tilde{t}},X-e_{\tilde{t}})> \lambda_{0} h^{M}(X, X).$
 \end{theorem}

{\bf Proof:} 
Using \eqref{4.4},  the $tt^*$-equation (\ref{alt-curv})  and 
$C_{\bar{X}}^{\flat}= \kappa C_{X}\kappa$, we can write
\begin{align}
\nonumber& h^M(R^M(X,\bar{X})X,X) = h^F(R^F(X,\bar{X})C_X,C_X)\\
\nonumber&=h^{end}(R^{D}(X,\bar{X}),[C_X, C_{\bar{X}}^\flat ])-h^{end}(A^F_X(C_X),A^F_X(C_X))\\
\label{ajut-curv}&=-h^{end}([C_X,C_{\bar{X}}^{\flat} ],[C_X,C_{\bar{X}}^\flat ]) -h^{end}(A^F_X(C_X),A^F_X(C_X)).
\end{align}
Relation 
(\ref{4.7}) follows. We now prove (\ref{4.8}). 
From $C_{e} = - \sf id$,  we need to show that $A^{F}_{e}=0.$ 
But  from the proof of Proposition 
\ref{canonical-TM},  $D_{e} (C_{X}) = C_{[e, X]}$. We obtain that
 $D_e(C_X)=D^F_e(C_X)$ and hence $A^F_e(C_X)=0$, as needed.
Part (a) is proved.

We now prove part (b).  For this, we need  to introduce some notation. 
Let $t\in M$ be such that $(T_{t}M , \circ_{t} , e_{t})$ is irreducible.
We choose a basis  
$\uuuu v=(v_1,...,v_n)$ of  $K_t$, in which 
$\kappa$, $g$ and $h$ take the standard form: $\kappa(\uuuu
v)=\uuuu v$,  $g(\uuuu{v}^t,\uuuu{v})={\bf 1}_n$ and 
$h(\uuuu{v}^t,\uuuu{v})={\bf 1}_n$. 
We extend this basis to  a basis of local sections 
$\uuuu{v}^{ext} = (v_{1}^{ext}, \cdots , v_{n}^{ext} )$  of $K$ on a neighbourhood  $\tilde{V}$ of $t$.
For any $X\in T_{\tilde{t}}M$ ($\tilde{t}\in \tilde{V}$)  we define the matrix 
$C^v_X$, by
$C_X (\uuuu{v}^{ext})=\uuuu{v}^{ext}\cdot C^v_X$,
i.e. $C_{X} (v_{i}^{ext} )= \sum_{j} (C_{X}^{v})_{ji} v_{j}^{ext}.$
When $X\in T_{t}M$,   
$C_{X}^{v}$ is symmetric, and, when  $X\in I^{max} (T_{t}M)$,  it is also nilpotent. 
We define the map 
$$
F: \tilde{V}\times  (M(n\times n, \mathbb{C})\setminus \{ 0\} ) \rightarrow \mathbb{R} ,\quad F(\tilde{t}, A):= 
-\frac{h^{end}_{\tilde{t}}([A,A^\flat ],[A,A^\flat ])}
{h^{end}_{\tilde{t}}(A,A)^{2}} ,
$$
where we identified $\mathbb{C}^{n}$  with $K_{\tilde{t}}$ 
using $\uuuu{v}^{ext}_{\tilde{t}}$. 
Thus $h_{\tilde{t}}$, $h_{\tilde{t}}^{end}$ are seen as  hermitian metrics on
$\mathbb{C}^{n}$ and $M(n\times n, \mathbb{C})$ respectively and the  superscript 
'$\flat$' denotes $h_{\tilde{t}}$ adjoint; in particular, $A^{\flat } = \bar{A}^{t}$ with respect to  $h_{t}$;
also $h^{end}_{t} (A, B) = \mathrm{tr} (A\bar{B}^{t})$, for any $A, B\in M(n\times n, \mathbb{C})$.
We denote by $\rho$ the restriction  $F\vert_{ \{ t \} \times ( M(n\times n, \mathbb{C})\setminus \{ 0\} )}$. It is given by
$$
\rho : M(n\times n, \mathbb{C})\setminus \{ 0\} \rightarrow \mathbb{R},\quad
\rho (A) = F(t, A) = - \frac{\textup{tr}   ([A,\bar{A}^{t}]\cdot\overline{[A,\bar{A}^{t} ] }^{t})} {\text{tr}(A\cdot   \bar{A}^{t} )^{2} }. 
$$
In order to define $V$ and the decomposition (\ref{dec-hyp}),  let
\begin{align*}
&\AAA:=\{A\in M(n\times n,\C),\ A\textup{ is symmetric and nilpotent}\},\\
&N:=\{A\in M(n\times n,\C),\  \textup{tr}(A\cdot\oooo{A}^t)=1\} .
\end{align*}
Remark that $\AAA \cap N$ is compact, $\AAA = \mathbb{R}^{\geq 0}\cdot   (\AAA\cap N )$ and
$\rho\vert_{\AAA\setminus \{ 0 \}}$ takes only negative values (if 
$\rho (A)=0$ then $[A, \bar{A}]=0$, so $\mathrm{Re}(A)$ and $\mathrm{Im}(A)$ commute; being symmetric, they are simultaneously 
diagonalisable; thus $A$ is diagonalisable and, being nilpotent, $A=0$;
this argument can be found in Lemma 4.2 of 
\cite{HS08}). Since $\AAA \cap N$ is compact, 
$F\vert_{ \{ t \} \times (\AAA \cap N)}= \rho\vert_{\AAA \cap N}$, is bounded from above by a negative number $- k_{0}$ 
(with $k_{0}>0$).  Choose an open neighborhood $\tilde{V}_{1}\times U_{1}$ of $\{ t \} \times (\AAA \cap N)$
in $\tilde{V}\times ( M(n\times n, \mathbb{C} )\setminus \{  0 \} )$ 
such that $F\vert_{ \tilde{V}_{1}\times U_{1}}$ is still bounded from above by  $- k_{0}$.
In particular, $U_{1}$ does not contain multiples of the identity matrix. 

Restricting $\tilde{V}_{1}$ if necessary, we extend  $I^{max} (T_{t}M)$ to a subbundle $\mathcal D \rightarrow \tilde{V}_{1}$ of $T\tilde{V}_{1}$, 
complementary to the line bundle generated by $e$.  We define  
\begin{align*}
V:=\{ \tilde{t} \in\tilde{V}_{1},\ C^v_X\in  \mathbb{R}^{>0} \cdot {U}_{1}\quad \forall  X\in {\mathcal D}_{\tilde{t}} \setminus \{ 0 \} \} 
\end{align*}
and we claim that it is a neighbourhood of $t$. To prove the claim, let $\{ s_{1}, \cdots , s_{n-1}\}$ be a basis of local sections of $\mathcal D$.
Define the map
$$
G : \tilde{V}_{1} \times (\mathbb{C}^{n}\setminus \{0  \} )\ni (\tilde{t}, X^{1}, \cdots , X^{n}) \rightarrow C^{v}_{\sum_{i} X^{i} s_{i}(\tilde{t}) } \in M(n\times n,\mathbb{C}).
$$ 
It satisfies $G(\tilde{t}, \lambda X^{1}, \cdots , \lambda X^{n} ) = \lambda G(\tilde{t}, X^{1}, \cdots ,  X^{n} )$, for any $\lambda , X^{i} \in \mathbb{C}$
(in fact, we only need $\lambda > 0$), and the image 
of its restriction to $\{ t\} \times  (\mathbb{C}^{n} \setminus \{ 0 \} )$ is included in $\mathbb{R}^{>0} \cdot  U_{1}$ 
(because ${\mathcal D}_{t} = I^{max}(T_{t}M)$, and, for any
$X\in I^{max} (T_{t}M)\setminus \{ 0 \}$, $C_{X}^{v} \in \mathbb{R}^{>0} \cdot 
({\mathcal A}\cap N ) \subset \mathbb{R}^{>0} \cdot U_{1}$).
A continuity argument shows that $V$ is a neighbourhood of
$t$.    We define $U_{\tilde{t}}$ as the set of all vectors $X\in T_{\tilde{t}} M$ such that $C^{v}_{X}\in\mathbb{R}^{\geq 0}\cdot {U}_{1}.$ 
One can easily see that decomposition (\ref{dec-hyp}) holds and has the required properties (i).

It remains to prove (ii) and (iii). 
The key observation in the proof of these statements is that 
$R^{sect}(X)$ (for any  $X\in T_{\tilde{t}} M \setminus \{ 0\}$) 
is the sum of a non-positive term and of 
$F(\tilde{t}, C^{v}_{X})$ (this follows from relation  (\ref{ajut-curv}) and the definition of $R^{sect}(X)$). 
For any $X\in U_{\tilde{t}}\setminus \{ 0 \}$, $C_{X}^{v}\in \mathbb{R}^{>0} \cdot U_{1}$. We obtain that 
$F(\tilde{t}, C^{v}_{X})$ 
(hence also $R^{sect}(X)$) 
is bounded from above by $-k_{0}.$ Thus (\ref{h1}) holds.
Claim (ii) follows.

Consider now $X:= e_{\tilde{t}} + Y\in T_{\tilde{t}} M$, where $Y\in U_{\tilde{t}}$.
Then
$$
F(\tilde{t}, C_{X}^v)  =  F(\tilde{t}, C_{Y}^{v} ) 
\cdot 
\frac{  h^{end}_{\tilde{t}} (C_{Y}^{v}, C_{Y}^{v})^{2}} {  h^{end} _{\tilde{t}} (C_{X}^{v}, C_{X}^{v})^{2}} .
$$
From (ii), the first factor from the r.h.s.  of this relation is bounded by $- k_{0}$. In the hypothesis of  (iii), the second factor is bigger than
$\lambda_{0}^{2} >0.$ Claim (iii) follows with $k_{1} := k_{0}\lambda_{0}^{2}.$
\hfill$\Box$

\medskip

\begin{remarks}\label{t4.7}
(i) The open neighborhood $V\subset M$ of the set
$\{t\in M\, |\, (T_tM,\circ_{t}, e_{t})\textup{ is irreducible}\}$
is almost a hyperbolic manifold. Only the direction of
the unit field with $R^{sect}(e)=0$ disturbs this.
Thus any smooth hypersurface in $M$ which is transversal to $e$
is a hyperbolic manifold. A related, but weaker situation turns up in the classifying
spaces of polarized Hodge structures. They are hyperbolic
in the horizontal directions \cite{GSch69}.
This has been generalized to classifying spaces for {\it
Brieskorn lattices} \cite{HS08}. A weaker property here 
(compared to \cite{GSch69,HS08}) is that $M$ and $V$ are representatives
of germs and not some global manifolds. 

\medskip

(ii) The main family of examples which we have in mind
comes from singularity theory.
Fix a holomorphic function germ $f:(\C^{n+1},0)\to (\C,0)$ with
an isolated singularity at 0 and Milnor number $\mu\in\N$.
The base space $M\subset \C^\mu$ of a universal
unfolding $F$ is a neighborhood of 0 in $\C^\mu$.
There is a natural $(TERP)(n+1)$-structure with unfolding condition
above $M$ (see \cite{Sa02}, \cite[ch. 8]{He03}, \cite{DS03}).
It comes from the Fourier-Laplace transform of the Gau{\ss}-Manin
connection of $F$. The $(TERP)$-structure
is not necessarily pure or polarized, so that one might not
obtain a $CV$-structure or just one with hermitian metric
$h$ with signature. But it turns out \cite[Corollary 11.4]{HS07}
that for any sufficiently big $r\in\R_{>0}$,
the $(TERP)(n+1)$-structure
of $r\cdot f$ is pure and polarized, so that one obtains
a $CV\oplus$ structure on the bundle $K(r\cdot f)\to M(r\cdot f)$.
Except from this good change of the $(TERP)$-structure,
this multiplying of $f$ by a scalar $r$ is harmless.
It does not change the topology and most invariants
of the singularity $f$.
Discriminant and caustic in $M$ are only rescaled. Therefore for most purposes, we can suppose that the base space
$M$ comes equipped with a natural $CV\oplus$-structure
with the unfolding condition.

\medskip

(iii) Even the mere existence of a {\it canonical} 
positive definite hermitian metric on the base space $M$
of a universal unfolding of a singularity,
has not been observed before.
It might become a very useful fact. 

\medskip
(iv) The base space $M$ of a universal unfolding of a singularity
comes with two interesting hypersurfaces, the discriminant
\begin{eqnarray}\label{4.10}
\DD=(\det E\circ)^{-1}(0)\subset M
\end{eqnarray}
and the caustic 
\begin{eqnarray}\label{4.11}
\KK=\{t\in M\, |\, (T_tM,\circ)\textup{ is irreducible}\}
=(\det H^{op}\circ)^{-1}(0).
\end{eqnarray}
Here $H^{op}\in\TT_M$ is a {\it socle field} (see \cite[14.1]{He02}
for the definition of socle fields), and  is not unique.
Its choice corresponds to a choice of a multiplication
invariant non-degenerate bilinear form on $TM$.
In the cases of the ADE-singularities, the complements
$M-\DD$ and $M-\KK$ are $K(\pi,1)$-spaces.

In 2006 Javier Fernandez de Bobadilla asked the second
author whether there are natural, positive definite, hermitian
metrics on $M-\DD$ and $M-\KK$, with  good properties which would
allow to prove that $M-\DD$ and $M-\KK$ 
are $K(\pi,1)$-spaces.  The following twist $h^{M,\DD}$ on $TM_{|M-\DD}$ is a canonical
hermitian metric on $M-\DD$, and the twist $h^{M,\KK}$ on $TM_{|M-\KK}$ 
(which depends on the choice of $H^{op}$) is at least a natural
hermitian metric on $M-\KK$,
\begin{eqnarray}\label{4.12}
h^{M,\DD}(X,Y)&:=& h^M((E\circ)^{-1}X,(E\circ)^{-1}Y)\\
h^{M,\KK}(X,Y)&:=& h^M((H^{op})^{-1}X,((H^{op})^{-1}Y).\label{4.13}
\end{eqnarray}
We expect that $h^{M,\DD}$ is complete on $U-\DD\subset M-\DD$ where
$U$ is a sufficiently small compact neighborhood in $M$ of a point on $\DD$,
and similarly for $h^{M,\KK}$ on $M-\KK$.
Unfortunately, the proof of Theorem \ref{t4.6}
does not give a negative sectional curvature $R^{sect}(\xi)<0$
for almost all directions for these twisted metrics.
Therefore Fernandez de Bobabilla's question is still open.

\medskip
(v) In the singularity case,
the almost hyperbolicity of $h^M$ on $M$ might be
useful for proving the conjecture about the variance
of the spectral numbers of a singularity
\cite[Conjecture 14.8]{He02}.
\end{remarks}

\end{document}